\documentclass[letterpaper,11pt,leqno]{article}
\usepackage{paper}
\usepackage[colorlinks=True, citecolor=blue, urlcolor=blue, linkcolor=blue, plainpages=false]{hyperref}
\usepackage{comment}
\newcommand{\blue}[1]{{\textcolor{blue}{#1}}{}}

\hypersetup{pdftitle={Minimalist LaTeX Template for Academic Papers}}

\begin{document}

\title{Childcare Waitlists}

\author{Yueyang Zhong, Jun Li, Senthil Veeraraghavan 
%
\thanks{Yueyang Zhong: London Business School (yzhong@london.edu). Jun Li: The University of Michigan Ross School of Business (junwli@umich.edu). Senthil Veeraraghavan: The Wharton School at the University of Pennsylvania (senthilv@wharton.upenn.edu). 
We thank Elizabeth Myers and Carrie Schultz (Bright Horizons), Gretchen Preston, Heidi McFadden, Harmony Smargon, and Morgan Peterson (Gretchen’s House), Christine Snyder (University of Michigan Child and Family Care), Deborah Stanley-McAulay and Angela Reese (Yale University Employee Engagement \& Workplace Culture), Priyank Arora (University of South Carolina), and Shiliang Cui (Georgetown University) for their helpful discussions and valuable insights.
}}

\date{First Version: October 1, 2026}   


\begin{titlepage}
\maketitle

Parents seeking childcare often face long waitlists. At the same time, vacancies at childcare centers can remain unfilled for extended periods.
We trace this paradox to two persistent features of childcare services. First, families consider multiple child care options due to enrollment uncertainties, 
resulting in \emph{phantom} listings that remain on waitlists even after families have found care elsewhere or no longer need a slot. Second, filling a slot is time-consuming: centers learn whether a family is actively seeking enrollment only after making an offer and waiting for a response. As families pursue alternatives while they wait and their children age, these delays and feedback effects amplify the mismatch between demand and available capacity.

We study two operational remedies. First, we examine regular verification that removes phantom listings before offers are made, improving the efficiency of enrollment and reducing vacancies. Second, we study mixed-age classrooms, in which some slots designated for older children can be flexibly used by younger children. This flexibility increases utilization, shortens waits, and reduces wasted offers. Its benefits extend beyond conventional pooling: as children move between age groups, classrooms can be filled through internal flows rather than costly external search, further reducing enrollment delays. Together, verification and mixed use offer a way to expand effective capacity in early childhood classrooms, where regulatory and operational constraints make adding physical capacity particularly costly. 


\end{titlepage}

\section{Introduction} \label{sec:intro}

Across much of the developed world, working families struggle to secure childcare. Demand is substantial. In the United States, three in five children under the age of five---approximately 13 million---spend 23 to 27 hours per week in nonparental care, and 66 percent of mothers with preschool-age children are employed, up from 8 percent in 1940~\citep{treasury2023}. Yet, high demand coexists with idle capacity. National surveys find that most U.S. centers turn families away for lack of space while also struggling to fill vacancies as they arise~\citep{brown2025}. Industry reports and state-level studies document families stranded on waitlists for months or years, while programs remain under-enrolled and licensed capacity goes unused~\citep{haspel2026paradox,weca2025}. The mismatch is not endemic to the United States. The Canada-based non-profit early childhood educational organization, \emph{Centres de la Petite Enfance}, reports contacting 100 to 150 families in Quebec to enroll only 20 to 25 children, even as parents wait a year or more~\citep{CityNews2025GhostParents}. In Europe, comparable shortages and long waits have prompted reforms to the allocation of scarce toddler slots~\citep{CarlssonThomsen2015ChildCareAllocation}. Across the post-industrial economies of the world, families face long waits while centers struggle to fill idle capacity, a phenomenon referred to as the \emph{childcare paradox}~\citep{haspel2026paradox}. 

The paradox arises because childcare waitlists are ledgers of past requests rather than current enrollment intent. Many families evaluate various options for childcare due to uncertainty in availability and preferences---enlisting at multiple childcare centers, hiring a nanny, arranging care with relatives, relocating, or changing employment plans. Thus, a listing at a childcare center may go unused. We refer to stale entries that go unenrolled as \emph{phantom listings}.\footnote{The term ``phantom'' was first introduced in the supply chain literature to describe a distortion in observed demand signals. When a product is rationed, customers may inflate their orders to secure larger allocations and later cancel these excess orders once supply constraints ease; \citet{lee1997bullwhip} refer to these as \emph{phantom orders}. As a result, the observed demand signal becomes increasingly disconnected from true demand and less informative as it propagates upstream through the supply chain.} Joining a waitlist is nearly costless, and withdrawal is seldom required. Hence, families have little incentive to update their status, especially as remaining on the list preserves a valuable backup option. Centers often observe the number of families on the list, but cannot distinguish active listings from phantom listings until they extend offers. This informational friction creates a gap between observed demand and enrollable demand. Each offer to a phantom listing takes time to resolve and consumes administrative effort, leaving a licensed slot vacant while an active family farther down the waitlist continues to wait. Thus, the very listings that make a center appear oversubscribed can be the reason for under-enrollment. 

This mismatch is evidently costly on all sides. For centers, the financial consequences can be immediate. In the United States, most for-profit childcare facilities operate on thin margins, and even one or two months below full enrollment can erase those margins~\citep[p. 15]{treasury2023}. On the other hand, families cannot reliably infer how long they will wait or when they should apply. Much research has indicated that such uncertainty has impacted female labor force participation and household fertility decisions. 
According to the US census bureau survey, an estimated 1.3 million people left the workforce to support a young child~\citep{Mader2022ChildCare}. As childcare services are quite decentralized, policymakers face challenges in addressing and resolving mismatches. Local authorities in the United Kingdom use waitlist data to assess unmet demand and determine where and for which age groups additional childcare provision is needed~\citep{welshgov2026childcare}.  Phantom listings often mask true unmet demand and thus misallocate expansion and public-investment decisions.

The goal of our paper is to develop an analytical understanding of how phantom listings distort the childcare enrollment process and generate the phantom paradox of long waits coexisting with persistent slot vacancies. We identify implementable operational policies to remedy the paradox---that is, to increase slot utilization while shortening families' waiting times.

We model the enrollment process for a single age-specific classroom with multiple licensed slots. Family requests arrive according to a Poisson process and join a waitlist. When a slot becomes available, the center contacts families on the waitlist until an offer is accepted, with each offer taking time to resolve. Because the center cannot observe whether a listed family still seeks care, offers to phantom families create a delay between slot availability and successful enrollment. We analyze this mechanism in two stages. The baseline model assumes that each listed family remains active with a constant probability, isolating the informational friction. We then allow this probability to decline with wait, capturing families' increasing likelihood of securing alternative care while they wait. 

We propose two operational remedies, based on theory and conversations with practitioners. The \emph{waitlist verification} process relies on eliciting families' current enrollment intentions and culling the waitlist accordingly. Families who have secured alternative care may still value their listing, so verifications can be imperfect. We characterize the optimal verification frequency and the minimum degree of accuracy required for verification to be beneficial. 
As a second policy we propose \emph{mixed-age classrooms} that convert selected age-dedicated slots into flexible capacity. Specifically, we direct some of the flexible capacity to be moved to more-expensive slots, and let the benefits of the policy accrue through learning, developmental milestones and age-based progression. 

\subsection{Our Main Findings} Our analysis yields three main research findings.  First, we show that the phantom paradox arises from the subtle interaction between enrollment uncertainties and the time-consuming offer resolution process, especially when patience is costly. Since enrollment is a ``high-touch'' process, and families take enrollment decisions seriously, centers need to make sequential offers. Each offer to a family on the waitlist can delay enrollment based on their response, and perpetuates the gap between waitlists and actual enrollments. The issue is persistent. Neither attracting more applicants nor adding moderate capacity can eliminate this informational friction. When families make other plans and turn into phantom listings due to impatience, the phantom paradox is further amplified: each successful enrollment requires even more offers; the enrollment rate slows, and slot utilization falls. Paradoxically, waiting times also decline because the center clears the waitlist increasingly by removing phantom families rather than enrolling children. 

We find that waitlist verification can shorten the waitlist and improve waitlist informativeness, by reducing wasted offers and idle capacity. The optimal verification frequency balances the marginal value of culling the phantom listings against the operational cost of verification. We characterize the optimal policy. 
Although the verification process can be costly, and often imperfect (families are often undecided on the best options for their child), verification can be worthwhile and can lead to significant cost savings and utilization improvement.
In our calibration, optimal verification reduces total monthly cost by approximately 13\%, increases slot utilization by about 3.5 percentage points, and lowers the prevalence of phantom listings relative to no verification. 

Finally, we offer an operational rationale for efficacy of mixed-age classrooms, often proposed due to their learning outcomes. In particular, investment towards lower low-age classrooms (compared to preschool classrooms) can increase overall slot utilization, shorten waiting times, and reduce phantom listings. The value of mixed-age classrooms extends beyond conventional pooling, which allows unmet infant demand to use otherwise idle toddler capacity without an extended search process. We find that converting even 2 slots in a classroom to be flexible mixed-age classroom can reduce operational costs by more than 15\%. 

\subsection{Related Literature} \label{sec:literature}

\noindent\textbf{Institutional Childcare Design.}
The childcare-assignment literature studies how scarce places should be allocated across families and providers. \citet{kennes2014daycare} formulate centralized public daycare allocation as a dynamic matching problem; \citet{reischmann2021deferred} design a deferred-acceptance mechanism that coordinates decentralized provider offers; and \citet{kuno2023strategic} studies strategic waiting and reapplication in a centralized daycare system. These papers take a submitted set of applicants, preferences, and priorities as the objects of allocation and ask who should receive which place and when. Most closely related in setting, \citet{brown2025} show why providers facing largely fixed classroom costs and stochastic demand may rationally maintain expected demand above capacity. We study a complementary operational mechanism. Conditional on prices, licensed capacity, and a provider's recorded waitlist, we analyze the vacancy-to-enrollment process when some applications no longer represent current demand. Time-consuming offers to these phantom applicants can leave staffed capacity idle even when the recorded list is long.

The institutional distinction is central to our contribution. Canonical school-choice mechanisms are centralized and organized around a common assignment cycle~\citep{gale1962,abdulkadiroglu2003}, while dynamic waitlist models design priority, lottery, or information policies for a roster whose membership is taken as current~\citep{arnosti2020housing,leshno2022dynamic,anunrojwong2023information}. In contrast, in the decentralized, provider-specific childcare setting that motivates our model, families apply independently to multiple centers, openings arise throughout the year rather than at a common clearing date, and no market-wide event refreshes every provider's waitlist. 

\medskip
\noindent\textbf{Abandonment and Hidden Customer Status.}
Our analytical model is related to service systems in which customers become unavailable before service. Classical abandonment models attach a patience time to each customer; when it expires, the customer leaves the system~\citep{zenios1999modeling,garnett2002designing}. For instance, \citep{su2004patient} consider a transplant waitlist model that allows refusals. 

In our setting, by contrast, a family's demand may lapse after it secures care elsewhere while its application remains recorded, and the center observes this change only when it makes contact. The closest analogues are ticket queues. In a ticket queue, a customer may leave without surrendering the ticket, so the provider discovers the absence only when the ticket is called \citep{jennings2016comparisons,kuzu2019wait}. Recent work considers processing absent tickets with some delay cost  \citep{kerner2026strategic} and capacity loss associated with disengagement \citep{castellanos2025silent}.
Ticket and contact-center interactions are generally short and repeatable. In contrast in childcare settings, time-scales are longer and invariably of economic and social consequence. A childcare offer may remain unresolved for days, leaving a regulated and staffed slot idle, whereas an accepted offer commits that slot for months. Furthermore, children age during the time waitlists remain unresolved, and list staleness is endogenous to waiting time. Our analysis needs to address hidden status and also its interaction with slowness of offer resolution. 

\medskip
\noindent\textbf{Verification and Schedule Management.}
A related operations literature examines unreliable participation in pending transactions and services. For instance, appointment systems use scheduling, reminders, and cancellation policies to mitigate no-shows \citep{green2008reducing,liu2010dynamic}, hiring-pipeline models optimize offers when candidate acceptance is uncertain \citep{epstein2024selection}, and models study the value of identifying absent customers before their turns \citep{hanukov2021benefits}. In much of the research, interventions are typically tied to one particular appointment, candidate, or imminent service opportunity. Childcare centers instead maintain long-lived registries while openings arise throughout the year. Our verification policy periodically audits the entire registry before vacancies occur, allows inaccuracy or concealing of status, and endogenizes the audit frequency. 

\medskip
\noindent\textbf{Process flexibility and Development Milestones.}
The mixed-age classroom policy is related to the expansive literature on process flexibility and resource pooling. Flexible capacity hedges against demand-capacity mismatch by allowing a common resource to serve multiple demand classes \citep{jordan1995principles,van1998investment}, and  managing flexible links \citep{bassamboo2012little}. Typically, demand classes are generally exogenous, and flexibility creates value by reallocating current capacity across them. In our setting, age classes are also connected by an endogenous, one-directional educational, intellectual and developmental transition. For instance, assigning a flexible slot to an infant not only serves current infant demand but can later generate a known-status move-up into the toddler classroom. 
Mixed-age flexibility therefore creates both conventional pooling value and a dynamic search-avoidance value. We thus extend the process-flexibility literature, with a tweak, to operations with progressive development in classes.

In sum, we examine childcare waitlists as a distinct operational design problem and show that they can be improved using better operations, information and capacity policies.

\section{A Model of Childcare Enrollment} \label{sec:model}

We model the enrollment process for a single age-specific classroom---infant, toddler, or preschool---at one childcare center. Centers manage each classroom's capacity separately because licensing rules, staffing requirements, and physical space constraints vary by age group.\footnote{For examples of state-specific rules governing age-based staff-to-child ratios, group sizes, and physical space, see Michigan Administrative Code R 400.8121 and R 400.8222 and Virginia Administrative Code 8VAC20-780-310 and 8VAC20-780-350 \citep{michigan2025childcare, virginia2026childcare}.} For example, an available infant slot cannot ordinarily be used to enroll a preschooler. We model one such classroom and its waitlist, tracing how an available slot becomes an enrollment.

\subsection{Licensed Slots, Family Requests, and the Offer Process} \label{subsec:primitives}

The enrollment process has three parts: licensed slots, family requests, and the offer process by which the center converts an available slot into a successful enrollment.

\paragraph{Licensed Slots.}
The classroom has $K$ licensed slots, each of which can be occupied by one enrolled child. An occupied slot becomes available when the enrolled child advances to the next classroom upon reaching the relevant age or developmental milestone, such as crawling, walking, or potty training, or when the family leaves the center for another reason, such as relocation or a change in parental employment. We use $\delta$ to represent how frequently either event makes one occupied slot available. A lower $\delta$ means that slots become available less often, and in other words, children remain enrolled longer, whereas a higher $\delta$ means that slots become available more often, and in other words, children remain enrolled for less time. Specifically, we assume that a child's enrollment duration is exponentially distributed
with rate $\delta>0$, 
so a child occupies a slot for $1/\delta$ time units on average. If all $K$ slots are occupied, slots become available at the aggregate rate $\mu := K\delta$, which is the fastest rate at which the classroom can generate openings.

\paragraph{Family Requests.}
Families typically contact childcare centers to request enrollment. New requests arrive at the center according to a Poisson process with rate $\lambda>0$. Each request is recorded on an uncapped waitlist. The center initiates an offer process whenever a slot is available. We consider families seeking the same type of slot.\footnote{In practice, families may need different childcare schedules, such as morning versus afternoon childcare or five-day versus three-day enrollment. The model in this paper considers one schedule and slot type.}

A family on the waitlist is \emph{active} if they still need childcare at the center and would accept an offer, and \emph{phantom} if they no longer seek childcare at the center and therefore would not accept one. We use the term phantom because their name remains on the waitlist even though they are no longer interested in enrolling. 
They may have enrolled at another center, hired a nanny, turned to a relative, relocated, or changed their work plans. Staying on the waitlist is often costless and, more importantly, preserves a valuable backup option. Families that no longer need childcare at the center therefore have little incentive to withdraw from the waitlist. Crucially, the center cannot observe a family's status---active or phantom---from its records and can only learn it by making contact.

Let $w \geq 0$ denote the time a family has spent on the waitlist. If the center contacts the family after $w$ units of time, they are active (that is, they accept the offer) with probability $p(w)$ and phantom (that is, they do not accept it) with probability $1-p(w)$. We assume that
\begin{equation} \label{eq:p_general}
    p(w) = p_0 \cdot e^{-\theta w}, \qquad p_0 \in (0,1],\; \theta \geq 0,
\end{equation}
where $p_0$ is the probability that a family would accept if contacted immediately after joining the waitlist. The acceptance probability decreases exponentially at rate $\theta$ because families are more likely to secure childcare through outside options as they wait longer. Even after securing outside childcare, they remain on the waitlist until the center contacts them.

\paragraph{Offer Process.} 

The offer process begins whenever a slot becomes available. The center selects a family from the waitlist and extends an offer. In practice, many centers operate sibling priority policies, affiliation or sponsor policies, community based policies, or need based priority policies alongside the order in which requests arrived. In this paper, we abstract from family preferences over slots and from the aforementioned priority categories, and assume a random-order-of-service (ROS) discipline throughout. The model and its main results, however, may extend to other work-conserving sequencing policies adopted by the center. 

Extending an offer does not necessarily lead to a successful enrollment. Families invariably take time to decide whether to accept, to decline, or to arrange a start date, and some never respond at all. In childcare centers, this decision typically takes days, so an offer remains unresolved for a period. We model this delay by assuming that the offer-resolution time is exponentially distributed with rate $\gamma>0$, so each offer takes $1/\gamma$ time units on average to reach a resolution.

If the family is active, they accept and the slot becomes occupied. If the family is phantom, the offer expires and is effectively wasted, in which case the center removes the family from the waitlist and turns to another one. The center incurs additional staff and administrative cost $s>0$ for a wasted offer. 

In Table~\ref{tab:notation}, we summarize the notation and its practical interpretation in childcare settings. 

\begin{table}[ht!]
  \centering
  \caption{Model notation and childcare interpretation.}
  \label{tab:notation}
  \small
  \begin{tabular}{@{}p{0.20\textwidth}p{0.75\textwidth}@{}}
    \toprule
    Notation & Interpretation in childcare settings \\
    \midrule
    \multicolumn{2}{@{}l}{\textit{Licensed slots and availability}}\\
    $K$ & Number of licensed slots in an age-specific classroom.  \\
    $\delta$ & Rate at which one enrolled child ages into the next classroom or leaves the center; hence, $1/\delta$ is the average time a child occupies a slot. \\
    $\mu=K\delta$ & Rate at which slots become available when all $K$ slots are occupied; equivalently, the classroom's maximum rate of slot availability.\\
    \midrule
    \multicolumn{2}{@{}l}{\textit{Family requests and status}}\\
    $\lambda$ & Average inflow of new family requests for childcare in a classroom. \\
    $w$ & Time a family has spent on the waitlist when the center contacts them. \\
    $p_0$ & Probability that a family would accept a slot if contacted immediately after joining the waitlist, before additional waiting creates further outside options. \\
    $\theta$ & Rate at which waiting families resolve their childcare need elsewhere while they remain on the waitlist. \\
    $p(w)=p_0e^{-\theta w}$ & Probability that a family contacted after waiting $w$ time units still needs childcare at the center and accepts the slot. \\
    $\alpha$ &  Steady-state average probability that a contacted family is active and accepts an offer.\\
    \midrule
    \multicolumn{2}{@{}l}{\textit{Offer process and waitlist maintenance}}\\
    $\gamma$ & Rate at which an offer is resolved through acceptance, refusal, or non-response; hence, $1/\gamma$ is the average elapsed time per offer before resolution. \\
    $s$ & Staff and administrative cost of an unsuccessful offer to a phantom family. \\
    \bottomrule
  \end{tabular}
\end{table}

\subsection{Enrollment Delays}\label{subsec:unique_feature}

Childcare enrollment has a distinct feature central to our analysis: a slot can remain open for an extended period without producing an enrollment. In many operations settings, such as call centers, newly available service capacity can often be put to use immediately. In childcare, however, a center must wait for each offer to be resolved (either accepted or declined) before the slot can be reassigned. Because offers made to phantom families on the waitlist do not lead to enrollment, they introduce delays between slot availability and successful enrollment. Consequently, despite a large number of families waiting, enrollment may require a sequence of offers, each adding further delay to the process. This delay leaves licensed slots vacant while active families further down the waitlist keep waiting.

\paragraph{Wasted Offers.}

To quantify the delay, consider a simple setting in which each contacted family accepts with the same probability $p$, independently of how long they have waited (i.e., $\theta=0$ in~\eqref{eq:p_general}). Filling one slot then requires $1/p$ offers on average. Of these offers, a fraction $(1-p)$ are made to phantom families and therefore do not result in enrollment. Hence, the expected number of wasted offers per successful enrollment is $(1-p)/p$, incurring a cost of $s \cdot (1-p)/p$ in wasted staff and administrative effort. 
In addition, each offer takes $1/\gamma$ time units on average to resolve. Thus, the average time from a slot becoming available to a child being enrolled is $1/(p\gamma)$.

Each slot therefore repeats a cycle with two stages. It is occupied for $1/\delta$ time units on average, and it then stands open for $1/(p\gamma)$ time units on average while the center works down the waitlist. Dividing the open stage by the length of the full cycle, the share of time that a slot stands open is
\begin{equation} \label{eq:open_share}
    \frac{1/(p\gamma)}{1/\delta + 1/(p\gamma)} \;=\; \frac{\delta}{p\gamma + \delta}. 
\end{equation}
We make two observations from~\eqref{eq:open_share}. First, the share is strictly positive whenever offers take time to resolve (i.e., $1/\gamma>0$), so children are enrolled more slowly than slots become available. Second, it is a property of one slot on its own and involves neither $K$ nor $\lambda$, so a center cannot shrink it by licensing a larger classroom or by attracting more requests.

Slots become available independently of one another, so a slot can open while the center is still running the offer process for a slot that opened earlier, and several offer processes can be under way at the same time. Applying~\eqref{eq:open_share} to each of the $K$ slots, the classroom has $K\delta/(p\gamma+\delta)$ open slots on average at any moment.

\paragraph{Waitlist Clearance.}

The waitlist is cleared as the status of the families on the waitlist is resolved through the offer process. During the process, the center learns the status of a family to be either phantom (in which case the family is removed from the waitlist) or active (in which case the child is enrolled). Each open slot has one offer outstanding, and each outstanding offer is resolved at rate $\gamma$. Thus, multiplying the average number of open slots by the rate at which each of them resolves offers, families leave the waitlist at the rate
\begin{equation} \label{eq:clearing_rate}
    \mu_{\mathrm{eff}} \;:=\; \gamma \cdot \frac{K\delta}{p\gamma + \delta} \;=\; \frac{K\gamma\delta}{p\gamma + \delta},
\end{equation}
which we call the \emph{effective clearance rate}. It counts every family the center takes off the list in the long run; in particular, a fraction $p$ of those removals are enrollments, and the remaining fraction $1-p$ are phantom families who have already arranged childcare elsewhere. The waitlist therefore shrinks far faster than children are enrolled, which is the reason the length of the waitlist says little about how many children the center actually enrolls. Children are enrolled at the rate $p\cdot\mu_{\mathrm{eff}} = \mu \cdot p\gamma/(p\gamma+\delta)$, strictly below the rate $\mu$ at which slots become available. This implies that the classroom enrolls fewer children than its licensed capacity and its turnover would allow.

Note that $\lambda$ does not affect the clearance process but plays a role in the size of the waitlist. When $\lambda < \mu_{\mathrm{eff}}$, the waitlist empties from time to time. We model the time from joining the waitlist to receiving an offer as an exponential random variable with rate $\mu_{\mathrm{eff}}-\lambda$. In a demand-rich childcare environment, $\lambda > \mu_{\mathrm{eff}}$, the waitlist often grows unboundedly. Families that face longer waiting times are more likely to resort to other outside options for their childcare need. In our model, this choice behavior corresponds to decreasing $p$ in $w$, an extension we discuss in \S\ref{subsec:impatience}. 

\noindent\textbf{An Illustrative Example.} We partner with a locally owned network of childcare centers in Michigan and use its operations to inform our parameter choices. A typical infant classroom has $K=10$ licensed slots, and children remain enrolled for approximately twelve months before moving up to the toddler room or leaving the center, so $\delta = 1/12$ per month. Families are given one week to respond to an offer before the center moves on to another family, so an offer is resolved in a quarter of a month on average and $\gamma = 4$ per month. Even with early offers, families may decline enrollment for various reasons. Families may prefer another center (and eventually withdraw), decline after a tour, or find that the fees, schedule, or start dates do not exactly meet their needs. Based on observations, we set the initial active probability to $p_0=0.38$. 

Filling one slot takes $1/p \approx 2.6$ offers, of which $(1-p)/p \approx 1.6$ are wasted on families who have already arranged childcare elsewhere. Because each offer takes a week to resolve, the classroom needs $1/(p\gamma) \approx 0.66$ months, less than three weeks, to turn an available slot into an enrollment. Equation~\eqref{eq:open_share} $\delta/(p\gamma+\delta) \approx 0.05$ implies that a slot is open $5$ percent of the time. Across the classroom, that amounts to $K\delta/(p\gamma+\delta) \approx 0.5$ of the ten licensed slots sitting vacant at any moment, each with an offer outstanding.

The waitlist clears at $\mu_{\mathrm{eff}} \approx 2.1$ families per month. Yet only $38$ percent of these removals result in enrollment, implying an enrollment rate of $p\cdot\mu_{\mathrm{eff}} \approx 0.79$ children per month (approximately $9.5$ per year), and the remaining $(1-p) \cdot \mu_{\mathrm{eff}} \approx 1.3$ monthly removals are phantom families.

\medskip

The example demonstrates the delay from a slot becoming available to a successful enrollment. Childcare centers incur significant costs as they work through long waitlists before finding a family that accepts an offer. In the next section, we show how even short delays in the offer process can generate and exacerbate a paradox commonly observed in childcare centers: families endure long waits even as childcare slots remain open.


\section{The Phantom Paradox in Childcare Services} \label{sec:paradox}

We begin with a simple version of the model in \S\ref{subsec:phantom_paradox}, in which the acceptance probability does not depend on how long a family has waited, so that $p(w) = p$ for all $w$, which corresponds to setting $\theta = 0$ in~\eqref{eq:p_general}. \S\ref{subsec:impatience} then restores the dependence on waiting time.


\subsection{Underenrolled Classrooms and Long Waitlists} \label{subsec:phantom_paradox}

\begin{lemma}\label{lemma:paradox}
    Suppose $p(w) = p$ for all $w$. 
    \begin{enumerate}[label=(\roman*)]
        \item \textbf{The classroom is underenrolled.} The steady-state fraction of occupied licensed slots is
        \begin{equation} \label{eq:umax}
            U \;=\; \min\left\{\frac{\lambda p}{K\delta},\; \frac{p\gamma}{p\gamma+\delta}\right\} < 1.
        \end{equation}
        A positive fraction of slots can be empty at any moment, even with strictly positive waitlist length. 
        \item \textbf{Families experience long waiting times.} The expected waiting time for a family from joining the waitlist to receiving an offer is
        \begin{equation} \label{eq:queue_length}
            \mathbb{W} \;=\; \frac{p\gamma+\delta}{(K\gamma \delta - p\gamma \lambda - \delta \lambda)^{+}},
        \end{equation}
        with the convention that $\mathbb{W} = \infty$ when the denominator vanishes. 
        Only a fraction $p$ of listed families are active, so the waitlist overstates the number of active families.
    \end{enumerate}
\end{lemma}
Lemma~\ref{lemma:paradox}(i) describes the fraction of open slots in the presence of family childcare requests. 
When the waitlist empties from time to time, every listed family is eventually made an offer and a fraction $p$ of them get enrolled. Each enrollment holds one of the $K$ slots for $1/\delta$ time units, which gives $U = \lambda p/(K\delta)$. When there are many requests, the waitlist does not empty out. Despite a large number of waiting requests, slots remain open because of the delays in the offer process. As shown in~\eqref{eq:open_share}, slots remain open for a fraction $\delta/(p\gamma+\delta)$ of time while waiting for enrollment to be completed. For the rest of the time, that is, $p\gamma/(p\gamma+\delta)$, slots are occupied. In both cases, slots are open for an extended period, so the licensed slot utilization $U$ is strictly less than $1$. 

Lemma~\ref{lemma:paradox}(ii) describes the expected wait time. The wait is governed by three distinct effects: how often a slot opens up, at rate $K\delta$; the time each family takes to resolve an offer, $1/\gamma$; and the number of phantom families who decline before an active one accepts, which depends on $p$. As more families request childcare slots (as $\lambda$ rises), the waitlist lengthens and can grow unboundedly.

The open slots and long waiting times observed in Lemma~\ref{lemma:paradox} lead us to our main proposition, explaining the phantom paradox often observed in the childcare sector. 

\begin{proposition}[Phantom Paradox]\label{prop:paradox}
    Empty slots and long waits coexist. The expected number of empty slots is $K\,(1 - U) \;=\; \max\left\{K - \frac{\lambda p}{\delta},\; \frac{K\delta}{p\gamma+\delta}\right\} \;\geq\; \frac{K\delta}{p\gamma+\delta} \;>\; 0$, and the expected waiting time until receiving an offer satisfies $\mathbb{W} \geq 1/\mu_{\mathrm{eff}} > 0$. As requests arrive more frequently, the expected waiting time grows unboundedly while the number of empty slots never falls below $K\delta/(p\gamma+\delta)$.
\end{proposition}

Proposition~\ref{prop:paradox} formalizes the simultaneous existence of open slots and waitlists, which has been widely described in the press and industry journals as the phantom paradox~\citep{haspel2026paradox, brown2025, weca2025}.  

To explain this paradox, we consider a center without phantom listings on the waitlist, which means every contacted family would accept the offer. There, the offer process is significantly shorter because of certainty in enrollment. For families, the waitlist length is truly informative of the expected time to receive an offer. Hence, a long waitlist can be interpreted as evidence of slot scarcity and a signal of true demand for childcare. 

The existence of phantom families severs the information link between the true demand and the waitlist. 
Furthermore, there is invariably a delay in confirming the status of families (active or phantom). Unfortunately, this delay compounds when there are multiple phantom families. As a result, slots remain open during the offer process.

As the demand for childcare grows, so does the number of phantom families on any waitlist.  The waitlist burgeons while the slot utilization increases only marginally, as the center has to go through multiple phantom families before a successful enrollment. Thus, an arbitrarily long waitlist can exist even when the classroom is not fully utilized. 

\noindent\textbf{Phantom Paradox: Numerical Illustration (Example Revisited).} We revisit the example in \S\ref{sec:model}. Recall that the effective clearance rate is $\mu_{\mathrm{eff}} \approx 2.1$ families per month. Suppose new requests arrive at rate $\lambda = 4.85$ per month (approximately weekly), which exceeds the clearance capacity more than twofold. The expected wait $\mathbb{W}$ therefore grows without bound. Slot utilization, however, is capped at $p\gamma/(p\gamma+\delta) \approx 0.95$: on average $9.5$ of the ten licensed slots are occupied and $0.5$ stand empty even as many families are on the waitlist ready for enrollment.
Our model explains how phantom entries on the waitlist can generate situations that are commonly described as a paradox: an unbounded waitlist coexists with idle classroom capacity. 


\begin{lemma}\label{lemma:comp_statics}
    Suppose $p(w) = p$ for all $w$ and $\mathbb{W} < \infty$. Then $\partial U/\partial\gamma = 0$. 
\end{lemma}



Lemma~\ref{lemma:comp_statics} highlights a \emph{practical} insight. Enforcing quicker decisions on offers 
does not necessarily translate into improved slot utilization at the center, as long as there is a delay between slot availability and enrollment. In other words, the phantom paradox persists. The intuition is as follows: In the stable regime, every active arrival is eventually enrolled, so faster offer resolution only brings enrollment forward but does not increase the volume of active demand served. 



\subsection{When Longer Waits Reduce Acceptance} 
\label{subsec:impatience}

So far, we have assumed that the probability $p$ that a family accepts a slot is constant. 
In practice, however, families that have waited for a longer period are more likely to have considered other options and secured care elsewhere and thus become phantom listings. Hence, we now allow the probability that a family accepts an offered slot to decrease with waiting time ($\theta > 0$ in~\eqref{eq:p_general}). This extension endogenizes phantom prevalence. Waiting times and acceptance probability are now intertwined. Our analysis reveals that the interaction between waiting time and phantom prevalence \emph{amplifies} the phantom paradox identified in \S\ref{subsec:phantom_paradox}.

The preceding analysis extends to wait-dependent acceptance by replacing the constant probability $p$ with its steady-state average counterpart. We now calculate the corresponding unconditional acceptance probability, denoted by $\alpha$, when families have different acceptance probabilities and are randomly picked from the waitlist.
When a slot becomes available, the center makes an offer to a family who has waited for $w$ units of time on the waitlist. The family accepts with probability $p(w) = p_0 \cdot e^{-\theta w}$. If the offer is declined, the center contacts another family, drawn at random from those still listed under the ROS discipline, so the waiting time of a contacted family is a draw from the steady-state distribution of $w$. In the long run, the fraction of offers that are accepted is characterized by the \emph{average acceptance probability} 
\begin{equation} \label{eq:alpha_def}
    \alpha \;:=\; \mathbb{E}_w\bigl[p(w)\bigr] \;=\; p_0 \cdot \mathbb{E}_w\bigl[e^{-\theta w}\bigr].
\end{equation}
Consequently, each enrollment requires $1/\alpha$ offers on average, of which $(1-\alpha)/\alpha$ are unsuccessful offers to phantom families.\footnote{Note that accept/reject outcome is governed by the elapsed wait time incurred by the family that is reached. Hence, successive offers need not be independent or identically distributed. The offer count $1/\alpha$ is the long-run time-average across all enrollments; see the proof of Theorem~\ref{thm:fixed_point}.} Because each offer takes $1/\gamma$ time units on average to resolve, the average time from a slot becoming available to a child being enrolled is $1/(\alpha \gamma)$. For any given $\alpha$, a slot behaves in the long run as in the baseline model with $p=\alpha$. In particular, the steady-state fraction of open slots is now $\delta/(\alpha\gamma+\delta)$, and the effective clearance rate is $\mu_{\mathrm{eff}}(\alpha) = K\gamma\delta/(\alpha\gamma+\delta)$. As in \S\ref{subsec:primitives}, the offer process depletes the waitlist at the long-run mean rate $\mu_{\mathrm{eff}}(\alpha)$. We therefore assume that the waiting time from joining the waitlist to receiving an offer is exponential with rate $\mu_{\mathrm{eff}}(\alpha)-\lambda$, and has infinite mean otherwise.

Similar to Lemma~\ref{lemma:paradox}, the following result characterizes slot utilization and expected waiting time. 

\begin{lemma}\label{lemma:alpha-u-wait}
For a given average acceptance probability $\alpha$, the steady-state fraction of occupied licensed slots and the expected waiting time for a family from joining the waitlist to receiving an offer are  
\begin{equation} \label{eq:alpha_u}
    U(\alpha) \;=\; \min\left\{\frac{\lambda\alpha}{K\delta},\; \frac{\alpha\gamma}{\alpha\gamma+\delta}\right\}, 
\end{equation}
and
\begin{equation} \label{eq:alpha_wait}
    \mathbb{W}(\alpha) \;=\; \frac{\alpha\gamma+\delta}{(K\gamma\delta - \alpha\gamma\lambda - \delta\lambda)^{+}}.
\end{equation}
\end{lemma}

Unlike in the constant-$p$ model, however, $\alpha$ is endogenous. Because $p(w)$ decreases with $w$, Equation~\eqref{eq:alpha_def} is a fixed-point equation: the steady-state distribution of $w$ is induced by $\alpha$, while $\alpha$ is itself determined by that distribution. 

\begin{theorem}\label{thm:fixed_point}
    The average acceptance probability $\alpha^*$ is the unique solution to~\eqref{eq:alpha_def}, equivalently, the unique solution to
    \begin{equation} \label{eq:alpha_fp}
        \alpha \;=\; \frac{p_0}{1 + \theta\,\mathbb{W}(\alpha)},
    \end{equation}
    and is strictly less than $p_0$.
\end{theorem}

Theorem~\ref{thm:fixed_point} shows that impatience strictly lowers the acceptance probability seen by the center. A family offered a slot immediately upon joining the waitlist accepts with probability $p_0$. As the wait lengthens, families are more likely to secure care elsewhere. Consequently, the equilibrium average acceptance probability is strictly below its initial level, namely, $\alpha^* < p_0$. This inequality underpins the amplified phantom paradox.

\begin{proposition}[Amplified Phantom Paradox] \label{prop:amplified_paradox}
    Relative to the constant-$p$ benchmark with $p=p_0$, wait-dependent acceptance shortens average waiting time but lowers slot utilization: 
    \begin{align*}
        \mathbb{W}(\alpha^*) \leq \mathbb{W}(p_0), \qquad U(\alpha^*) < U(p_0).
    \end{align*}
\end{proposition}

Proposition~\ref{prop:amplified_paradox} shows that the average time families spend on the waitlist becomes shorter when the acceptance probability decreases with waiting time. While this is an improvement for families that manage to get enrolled, this improvement comes at the expense of many more phantom listings: those phantom families on the waitlist are no longer active and end up with (possibly worse) outside options. Due to the increased number of phantom listings, the slot utilization at the center worsens. Thus, the decline in waiting time for the enrolled families does not necessarily indicate improved waitlist management for the center. Since $\alpha^*<p_0$, each successful enrollment at the center now requires contacting more families on average ($1/\alpha^*>1/p_0$), increasing both the effort and the time required to fill an available slot. Hence, waitlist impatience reduces slot utilization, namely, $U(\alpha^*) < U(p_0)$. 

Endogenous impatience thus amplifies the phantom paradox: the center enrolls fewer children, leaves more slots vacant, and wastes more offers on phantom families. We illustrate this amplification using our example. 

\noindent\textbf{Amplified Phantom Paradox: Numerical Illustration (Example Revisited).}
We revisit the Michigan infant-classroom example with $K=10$, $\delta=1/12$, $\gamma=4$, $\lambda=4.85$, and $p_0=0.38$, and introduce wait-dependent acceptance by setting $\theta=0.075$ per month. Under this specification, a family's acceptance probability falls by half after approximately $\ln(2)/\theta=9$ months on the waitlist. Solving the fixed-point equation~\eqref{eq:alpha_def} yields an equilibrium average acceptance probability of $\alpha^*\approx0.15$, well below $p_0$. Each enrollment now requires about $6.7$ offers, the classroom enrolls $0.72$ children per month, and the waitlist holds about $100$ names, $85$ percent of them phantom, with a family waiting $21$ months on average before it is contacted. The enrollment rate matches that of our partner network's infant rooms, which enroll roughly $0.75$ children per month, and the offer count is of the same order as the eight or so offers they make per enrollment. Table~\ref{tab:waitdep_comparison} compares system performance under the constant-$p$ benchmark and wait-dependent acceptance. We note that the numerical observations mirror Proposition~\ref{prop:amplified_paradox}.

\begin{table}[ht]
\centering
\caption{Wait-dependent acceptance worsens system performance metrics. Parameters are $K=10$, $\delta=1/12$, $\gamma=4$, $\lambda=4.85$, $p_0=0.38$, and $\theta = 0.075$.}
\label{tab:waitdep_comparison}
\small
\begin{tabular}{@{}lrrr@{}}
    \toprule
    Metric 
    & Constant $p$ 
    & Wait-dependent acceptance 
    & Change \\
    \midrule
    Average acceptance probability 
        & 0.38 
        & 0.15 & $-61\%$ \\
    Offers per enrollment 
        & 2.6 
        & 6.7 & $+150\%$ \\
    Slot utilization $U$ 
        & 0.95 
        & 0.87 & $-8\%$ \\
    Enrollment rate $K\delta U$ (per month)
        & 0.79 
        & 0.72 & $-8\%$ \\
    Average number of empty slots $K(1-U)$ 
        & 0.5 
        & 1.3 & $+150\%$ \\
    \bottomrule
\end{tabular}
\end{table}


Before we explore potential ways to address the phantom paradox, we conclude the section with some insightful comparative statics. 
\begin{lemma}\label{lemma:alpha_cs}
    $\partial\alpha^*/\partial\lambda \leq 0$, $\partial\alpha^*/\partial\theta \leq 0$, and $\partial\alpha^*/\partial\gamma \geq 0$. Moreover, $\partial \mathbb{W}(\alpha^*)/\partial \gamma \leq 0$ and $\partial U(\alpha^*) / \partial \gamma \geq 0$.
\end{lemma}

Lemma~\ref{lemma:alpha_cs} shows how acceptance probabilities shift with market conditions. 
The average acceptance probability decreases when there are more family requests, outside options are easier to secure, or offers are resolved more slowly. Higher demand and slower offer resolution lengthen waiting times, while more accessible outside options increase the likelihood that families secure care elsewhere before receiving an offer. Compared to the constant-$p$ case, more accessible outside options can shorten waiting times but decrease slot utilization. 

\subsection{Policies to Remedy the Phantom Paradox 
} \label{sec:policies} 

Our analysis reveals how the informational friction can lead to the phantom paradox. A center learns whether a listed family still wants a slot only by extending an offer, each offer takes time to resolve, and every offer made to a phantom listing holds a slot open without filling it. \S\ref{subsec:phantom_paradox} shows that the resulting idle capacity cannot be eliminated simply by attracting more requests or by licensing more slots. \S\ref{subsec:impatience} reveals that impatience further widens the gap, as the center wastes more offers on a list containing a higher share of phantom families.  We consider remedies informed by our model and practitioner conversations: learn active status of families before extending time consuming offers, and allow idle slots produced by the paradox in one classroom to serve families waiting for another classroom.  We propose two operational policies along these lines to counteract the inefficiency caused by phantom listings. 


The first, \emph{waitlist verification}, falls under a broad collection of actions that update the waitlist before the center makes offers. 
The second set of policies reconfigures the capacity allocation within the center (without touching the waitlist). Specifically, we focus on the strategy of \emph{mixed-age classrooms}, which reconfigure \emph{capacity} across age groups so that one group's idle slots can relieve another group's waitlist.

Both policies are implementable in practice, with the first policy acting on the informational friction and the second acting on the idle capacity. Another appeal is that neither relies on potentially controversial forms of differential treatment, such as price discrimination or priority classes based on family characteristics.

We analyze both policies under the same objective, which combines three sources of inefficiency---cost of wasted offers, cost of idle capacity, and cost of policy implementation: 
\begin{align} \label{eq:section6_cost}
    \mathrm{TC} \;=\; \underbrace{s \times (\text{offers to phantom listings})}_{\text{wasted offers}} \;+\; \underbrace{r \times (\text{empty licensed slots})}_{\text{idle capacity}} 
    \;+\; \underbrace{(\text{operational cost})}_{\text{policy}},
\end{align}
where $r$ is the contribution a filled slot earns, so that every slot standing empty costs the center $r$ per unit time in forgone margin. 

\section{Remedial Policy: Waitlist Verification} \label{sec:reverify}

Our model so far captures the fact that a center learns a family's status (active/phantom) only by making an offer, and that it makes offers only when slots open. The status of a family remains invisible until an offer is made to them. Without the status information, changing the order of contact only reshuffles who is reached first. 

One direct remedy is to verify families' status proactively rather than waiting until a slot opens and an offer is made. The center periodically contacts every listed family, asks whether they still want a slot, and removes those who say they do not. Periodic verification is often seen in allocation systems with high centralization or subject to formal oversight, such as affordable housing and organ registries \citep{arnosti2020housing, zhang2010sound, zenios1999modeling}. In contrast, the childcare sector is much less regulated and fairly decentralized. 

The key advantage of the verification policy is that the center verifies the status of listed families in advance according to its own schedule, rather than contacting families after the slot opens. As a result, when the slot opens, the center can enroll the verified families without the slot remaining idle for too long. Thus, the verification policy, as we proposed, is broader. For example, our verification policy allows for advance planning when departure/graduation of a child is expected. Suppose a center knows a child will age up in March. Then, it can begin its verification process a few months in advance (e.g., in January) and allot the slot for a pre-verified enrollment in March when the slot opens. 

Waitlist verification has two complications. The first is that verification is costly. Every listed family must be contacted, whether or not they turn out to be phantom in the future. Every such contact incurs staff and administrative cost, so the center must choose how often to ask families.

The second complication, as observed by practitioners, is that the responses can be unreliable (in the sense that the responses from the contacted families may not reflect their true status). Families may find being on the waitlist still valuable even if they do not wish to enroll immediately. For instance, a family that has arranged temporary childcare elsewhere still values a place on the waitlist, in case that arrangement fails or preferences change. For a family with uncertainty when it comes to care, saying no ahead of time forfeits a position on the waitlist that is free to hold. 

We therefore assume that asking does not always reveal the true status. In each round, families that are active on the waitlist confirm they are active. Families that are phantom confirm that they no longer need a spot at the center with probability $\phi \in (0,1]$, and otherwise remain listed. 
The benchmark $\phi = 1$ represents perfect revelation, and small $\phi$ describes a waitlist that cannot be cleared by asking families ahead of offers.

The center runs a verification round every $T > 0$ time units and incurs a cost of $c_{\nu} > 0$ for each family contacted. Write $\nu := 1/T$ for the verification frequency, so a larger $\nu$ means more frequent rounds. Since only a fraction $\phi$ of contacted phantom families is revealed, 
verification removes a listed phantom family at a rate of $\phi\nu$. The verification process operates as a second channel along with the offer process for curating the waitlist. Since $\phi \leq 1$, this curation is only partially effective. 

Throughout this section, we assume $\lambda<K\gamma$, ensuring finite wait, and $\lambda p_0 > K\gamma\delta/(\gamma+\delta)$, ensuring sufficient demand even as verification becomes arbitrarily frequent. Both inequalities hold in the numerical illustration.

\subsection{The Equilibrium Acceptance Probability Under Verification}

Suppose the waitlist is verified at frequency $\nu>0$, the center wastes fewer offers for each slot that opens. Let $\alpha(\nu)$ be the corresponding average acceptance probability and $\mathbb{W}(\nu)$ be the corresponding expected waiting time until a listed family receives an offer.\footnote{$\mathbb{W}(\nu)$ differs from the offer-process-only waiting time in~\eqref{eq:alpha_wait}, as it incorporates both channels of waitlist curation: offers and verification.} Then, every listed family receives an offer at rate $1/\mathbb{W}(\nu)$. When verification is more frequent (larger $\nu$), the number of families remaining on the waitlist decreases (as more phantom families are removed). However, the remaining families expect a shorter wait and may accept an offer with higher probability. Therefore, the equilibrium acceptance probability $\alpha^*(\nu)$ is a result of these two forces---reduced number of families remaining on the waitlist and increased acceptance probability of the remaining families. 

\begin{theorem}\label{thm:alpha_nu}
    For a verification frequency $\nu \geq 0$ and $\phi \in (0,1]$, the average acceptance probability $\alpha^*(\nu)$ is the unique solution to 
    \begin{equation} \label{eq:alpha_nu}
        \alpha \;=\; \frac{p_0\,\bigl(1 + \phi\nu\,\mathbb{W}(\nu)\bigr)}{1 + (\theta + p_0 \phi \nu) \, \mathbb{W}(\nu)}, 
    \end{equation}
    where $\mathbb{W}(\nu) \in (0,\infty)$ is the unique solution to 
    \begin{align}\label{eq:capacity_const-alt}
        \frac{\lambda p_0}{1+\theta \mathbb{W}(\nu)} \left(\frac{1}{\delta} + \frac{1}{\gamma}\right) + \frac{\lambda (1-p_0+\theta \mathbb{W}(\nu))}{(1+\theta \mathbb{W}(\nu))(1+\phi\nu\mathbb{W}(\nu))} \cdot \frac{1}{\gamma} = K. 
    \end{align}
    Furthermore, $\mathbb{W}(\nu)$ is strictly decreasing in $\nu$, and $\alpha^*(\nu)$ is strictly increasing in $\nu$, with $\alpha^*(0) = \delta(K\gamma-\lambda)/\lambda\gamma$ and $\lim_{\nu\to\infty}\alpha^*(\nu) = 1$.  
\end{theorem}


The verification policy improves the average acceptance probability $\alpha^*(\nu)$, mainly by actively pruning the phantom families and changing the composition of the waitlist. We characterize the composition of the waitlist for any verification policy in Lemma~\ref{lemma:reverify_composition}. 
Let $Q_a(\nu)$ and $Q_p(\nu)$ denote the equilibrium number of active and phantom families on the waitlist. 
Since every listed family is equally likely to be reached by an offer, whatever its status, the average acceptance probability equals the chance that an offer reaches an active family on the waitlist, which can be written as $\alpha^*(\nu) = Q_a(\nu)/(Q_a(\nu) + Q_p(\nu))$. 

\begin{lemma}[Waitlist Composition and Slot Utilization] \label{lemma:reverify_composition}
    Fix a verification frequency $\nu \geq 0$ and $\phi \in (0,1]$. The expected numbers of active and phantom families on the waitlist are
    \begin{equation} \label{eq:AP_nu}
       Q_a(\nu) = \frac{\lambda p_0  \mathbb{W}(\nu)}{1+\theta \mathbb{W}(\nu)}, \qquad  Q_p(\nu) = \frac{\lambda \mathbb{W}(\nu) (1-p_0 + \theta \mathbb{W}(\nu))}{(1+\theta \mathbb{W}(\nu))(1+\phi\nu \mathbb{W}(\nu))},
    \end{equation}
    and the steady-state fraction of occupied licensed slots is 
    \begin{align}\label{eq:U}
        U(\nu) = \frac{\lambda p_0}{K\delta} \frac{1}{1 + \theta \, \mathbb{W}(\nu)}. 
    \end{align}
    Both $Q_a(\nu)$ and $Q_p(\nu)$ strictly decrease in verification frequency $\nu$, with $\lim_{\nu\to\infty} Q_p(\nu) = 0$, while $U(\nu)$ strictly increases in $\nu$. 
\end{lemma}

Lemma~\ref{lemma:reverify_composition} captures the essence of the verification policy: verification makes waitlists shorter and more informative. Verification removes only the phantom families from the waitlist.  Consequently, the expected number of phantom families $Q_p$ falls. 
Surprisingly, the expected number of active families $Q_a$ also falls, although they remain on the waitlist. This is because the verification process makes the enrollment process faster. Recall that, without the verification process, the center had to contact many families (that are phantom) before enrolling an active family. Due to verification, the center is able to reach an active family sooner (due to a larger proportion of active families on the waitlist), which reduces the expected wait times among families who eventually enroll. As a result of the faster enrollment process, the slot utilization, shown in~\eqref{eq:U}, also improves.

Earlier, we noted that long waitlists reduce enrollment commitment among families, leading to amplification of the phantom paradox in \S\ref{subsec:impatience}. This is because families who wait long without hearing back have to make alternative arrangements. 
The verification policy addresses this problem by shortening the waitlists. 
In practice, verification policy has the additional advantage of improving service through continued communication with the active families.  

Another lever to improve waitlist management is to improve the accuracy of the verification, i.e., a larger $\phi$. Note that a larger $\phi$ (that is, more verification accuracy) has an equivalent effect to a larger $\nu$ (that is, higher verification frequency). 


\subsection{The Optimal Verification Process}

The center incurs three sources of costs---wasted-offer cost, idle-capacity cost, and verification process cost. 
All three costs change with the verification frequency $\nu$. Offers still reach phantom families, though less frequently, at rate $Q_p(\nu)/\mathbb{W}(\nu)$, each costing $s$. Each empty slot costs the center $r$ per unit time, and verification reduces the number of such slots. And each round contacts every listed family at cost $c_{\nu}$, regardless of the accuracy. The center chooses a verification frequency $\nu \geq 0$ to minimize
\begin{equation} \label{eq:reverify_cost}
    \mathrm{TC}(\nu) \;=\; \underbrace{\frac{s\,Q_p(\nu)}{\mathbb{W}(\nu)}}_{\text{wasted offers}} \;+\; \underbrace{r K (1-U(\nu))}_{\text{idle capacity}} + \underbrace{c_\nu\,\nu\,\bigl(Q_a(\nu) + Q_p(\nu)\bigr)}_{\text{verification}}.
\end{equation}

\begin{proposition}[Optimal Verification Frequency] \label{prop:optimal_review}
    An optimal verification frequency $\nu^*$ exists and is unique.
    \begin{enumerate}[label=(\roman*)]
        \item Verification is beneficial, that is, $\nu^*>0$, if and only if the verification process is sufficiently accurate, namely, 
            \begin{equation} \label{eq:review_threshold}
                \phi \;>\; \phi_{\min} \;:=\; \frac{c_\nu}{V\,(1-\alpha^*(0))},
            \end{equation}
            where $V \;:=\; s + (r+s\delta) / \gamma$ is the value of removing one phantom family from the waitlist (that is, the wasted offer it saves, plus the occupancy it releases, each additional occupied slot earning $r$ and averting $\delta$ wasted offers per unit time). Otherwise, $\nu^*=0$.
        \item The optimal verification frequency $\nu^*$ is higher when the wasted-offer cost $s$ increases, the idle-slot cost $r$ increases, and the per-contact cost $c_\nu$ decreases. These comparative statics are strict while $\nu^* > 0$. 
    \end{enumerate}
\end{proposition}

Proposition~\ref{prop:optimal_review}(i) highlights that, when the verification process is sufficiently accurate---in the sense that families reveal their status candidly---it provides two benefits. It improves the informativeness of the waitlist, leading to fewer wasted offers and less slot idle time. The improvement comes at the cost of verification. At the optimal verification frequency, the value of removing one phantom family from the waitlist is balanced against the cost of one verification.  

Proposition~\ref{prop:optimal_review}(ii) further implies that the center should verify more frequently when vacant slots entail greater forgone revenue. Moreover, when the cost of an idle slot largely exceeds that of a wasted offer, the center should also verify more frequently as the cost imposed by phantom families increases or the cost of contacting families decreases.

\paragraph{Numerical Illustration.}

We return to the same classroom example in \S\ref{sec:paradox}, with $K = 10$, $\delta = 1/12$, $\gamma = 4$, $\lambda = 4.85$, $p_0 = 0.38$, $\theta = 0.075$, and $s = \$50$. We set the per-contact cost at $c_\nu = \$3$, reflecting an automated message and the staff time to process the replies, and value a filled slot at $r = \$1{,}200$ per month, the tuition an infant place earns net of the child's variable cost. Table~\ref{tab:reverify} reports the equilibrium as the verification frequency changes. The least costly of the tabulated frequencies is highlighted in blue.

\begin{table}[ht]
\centering
\caption{Outcomes under various verification frequencies $\nu$ when fully accurate ($\phi = 1$).
Parameters: $K = 10$, $\delta = 1/12$, $\gamma = 4$, $\lambda = 4.85$, $p_0 = 0.38$, $\theta = 0.075$, $s = \$50$, $c_\nu = \$3$, $r = \$1{,}200$.  All costs are monthly.}
\label{tab:reverify}
\small
\begin{tabular}{@{}l r r r r r r r r@{}}
    \toprule
    & & Waitlist & Slot & Wasted & Idle & Verifi- & Total & \\
    Verification freq. $\nu$ & $\alpha^*$ & length & Util. & offers & cap. & cation & cost & Savings \\
    \midrule
    None ($\nu=0$)                 & 0.15 & 98 & 0.87 & \$206 & \$1{,}455 & \phantom{00}\$0 & \$1{,}661 & --- \\
    Annual ($\nu = 1/12$)          & 0.33 & 43 & 0.94 & \phantom{0}\$81 & \phantom{0,}\$722 & \phantom{0}\$11 & \phantom{0,}\$814 & 51.0\% \\
    Twice yearly ($\nu = 1/6$)     & 0.44 & 32 & 0.95 & \phantom{0}\$52 & \phantom{0,}\$548 & \phantom{0}\$16 & \phantom{0,}\$616 & 62.9\% \\
    Quarterly ($\nu = 1/3$)        & 0.57 & 24 & 0.96 & \phantom{0}\$30 & \phantom{0,}\$421 & \phantom{0}\$24 & \phantom{0,}\$475 & 71.4\% \\
    Every two mon. ($\nu = 1/2$) & 0.66 & 21 & 0.97 & \phantom{0}\$21 & \phantom{0,}\$369 & \phantom{0}\$32 & \phantom{0,}\$422 & 74.6\% \\
    \blue{Monthly ($\nu = 1$)} & \blue{0.78} & \blue{18} & \blue{0.97} & \blue{\phantom{0}\$11} & \blue{\phantom{0,}\$311} & \blue{\phantom{0}\$53} & \blue{\phantom{0,}\$375} & \blue{77.4\%} \\
    Fortnightly ($\nu = 2$)        & 0.88 & 16 & 0.98 & \phantom{00}\$6 & \phantom{0,}\$279 & \phantom{0}\$94 & \phantom{0,}\$379 & 77.2\% \\
    Weekly ($\nu = 4$)             & 0.93 & 15 & 0.98 & \phantom{00}\$3 & \phantom{0,}\$262 & \$176 & \phantom{0,}\$442 & 73.4\% \\
    \bottomrule
\end{tabular}
\end{table}

Among the tabulated frequencies, the best policy is to verify the waitlist monthly, which cuts total cost by $77\%$, compared to no verification; the optimal frequency of Proposition~\ref{prop:optimal_review}, $\nu^* \approx 1.3$ rounds a month, saves slightly more ($78\%$). The cost reduction comes from two sources, as discussed above: wasted-offer costs fall by $\$195$ a month and idle-capacity costs fall by $\$1{,}144$ a month. On the other hand, monthly verification incurs a cost of $\$53$ per month. Even a single round a year recovers most of the gain, because the unverified list is $85$ percent phantom and any round removes them; the returns to further rounds then diminish.


With monthly verification, slot utilization improves from $0.88$ to $0.97$, an $11\%$ increase over the no-verification baseline, and the average number of empty slots falls from $1.2$ to $0.3$. Verification also reduces the share of phantom families from $85\%$ to $22\%$ and the waitlist from about $98$ names to $18$, so that a family receives an offer after $6.5$ months on average rather than $20.2$. Verifying waitlists too often (e.g., weekly) is counterproductive: the value of removing the few remaining phantom families is no longer commensurate with the cost of contacting every listed family.

\noindent\textbf{Effect of Accuracy.} 
The verification accuracy level changes the magnitude of the outcomes but does not change the nature of trade-offs embedded in the verification policy. 
At monthly rounds, the saving falls from $77\%$ under full accuracy to $76\%$ at $\phi = 0.75$, $73\%$ at $\phi = 0.5$, and $65\%$ at $\phi = 0.25$. Verification remains worthwhile even when three in four phantom families choose to stay on the waitlist even when they no longer need childcare at the center; at $\phi = 0.05$ monthly rounds still save $32\%$. As long as the accuracy is above a certain threshold $\phi_{\min}$, identified in Proposition~\ref{prop:optimal_review}(i), the waitlist verification process is beneficial; in this calibration the threshold is about $0.01$. 

The verification policy cleans the waitlist but does not affect the center's capacity. We now discuss another policy recommendation that reconfigures capacity based on regulations and children's developmental milestones. 


\section{Remedial Policy: Mixed-Age Classrooms} \label{sec:pooling}

Most childcare centers partition capacity by age because staffing requirements, group-size limits, physical-space rules, and children's developmental needs differ across age groups. Typically, centers operate separate classrooms for infants, toddlers, and preschoolers. 

Among the age groups, the need for childcare support is acute among families with infants~\citep{nces2022childcarereasons}. 
Unfortunately, waiting times for infant care are precisely the longest~\citep{hanson2024earlychildhood}. This is not only due to the high volume of infant care demand, but also due to shortages in capacity. In particular, regulatory and developmental requirements make infant care substantially more costly to provide and operationally more intensive than care for older children~\citep{hotz2011,gorry2017regulation,jessenhoward2018infanttoddlerdeserts}. Regulatory requirements---including staff-to-child ratios, group-size ceilings, and physical space requirements---are substantially more stringent for infant care than for care in older-age classrooms, such as toddler and preschool classrooms. For instance, Michigan requires infant and toddler classrooms (typically three to thirty months, with some exemptions) to have at least one adult caregiver per four children. In contrast, classrooms for older children have relaxed staff-to-child ratios. Specifically, for children from thirty months to three years, the ratio is 1:8, and for children above three years, the ratio is even smaller at 1:10~\citep{michigan2025childcare}. 
Thus, it is economically rational for centers to invest in older-age classrooms. In practice, many childcare centers avoid serving infants and toddlers and concentrate on older children, for whom tuition revenues are more likely to cover the required staffing and facility costs~\citep{jessenhoward2018infanttoddlerdeserts}.

Our proposed mixed-age classroom policy recommends shifting capacity toward the more expensive younger-age classrooms. This recommendation pushes capacity in a counterintuitive direction. However, we demonstrate that the benefits from improved utilization and shorter waitlists can be substantive.  

Mixed-age grouping is used in some early childhood education settings, including the Head Start programs and Montessori early education centers~\citep{OfficeofHeadStart2025HeadStartPrograms,demangeon2023montessori}, but evidence on its developmental effects is mixed and varies by children's age and peer composition~\citep{bell2013agecomposition,ansari2016agecomposition,foster2020agecomposition,vanlombergen2025mixedage}. Our model provides analytical support for the operational benefits of implementing such mixed-age grouping in early childhood.


The operational benefits of the mixed-age classroom policy are subtle and rely on the child's developmental progress with age.
The first benefit is due to operational pooling---our policy creates additional capacity for infant demand, which shortens infant waitlists and raises the infant classroom acceptance probability. However, the benefits of the mixed-age classroom policy extend beyond capacity pooling effects. 
The second benefit stems from expediting the enrollment process.  Specifically, infants in the center can directly move up to the toddler classroom. This process saves the center staffing costs associated with searching for a new family to fill a toddler slot. The third practical benefit of mixed-age classrooms is that the center is aware of a child's developmental stage and can transition the child into the next level when appropriate.

\subsection{Modeling Growth Milestones and Class Transitions}

Consider a center that runs two age-based classrooms, infants ($I$) and toddlers ($T$). Class $j \in \{I,T\}$ has $K_j$ licensed slots, and each receives age-based requests at rate $\lambda_j$.  Each occupied classroom slot becomes available at rate $\delta_j$. The turnover rates for the two classes differ because children's developmental milestones differ. Further, different states have different age cutoffs (see, e.g., \cite{michigan2025childcare, virginia2026childcare}). For simplicity, we assume that the two classes share the offer-resolution rate $\gamma$, the initial acceptance probability $p_0$, and the decay rate $\theta$. Each class runs the offer process on its own waitlist, so class $j$ has its own average acceptance probability $\alpha_j^*$.

A slot opens for two reasons: when an enrolled child reaches the developmental milestone and moves up to the next classroom, or when a family leaves the center for an unrelated reason such as relocation. 
We denote by $\sigma \in [0,1]$ the probability that an infant moves up to the toddler classroom. Then, $1-\sigma$ denotes the probability of discontinuing enrollment at the focal center upon completion of the infant stage. 
Since the child is already enrolled, the move-up process is straightforward without the need for an extended offer process. 

We denote the rate of moving up by $g$ (in practice, a child moving up to the next class is often referred to as \textit{graduation}). 
In steady state, infant enrollments and infant departures balance, so the move-up rate can be written as 
\begin{equation} \label{eq:grad_flow}
    g \;=\; \sigma\,\lambda_I\,\alpha_I^*. 
\end{equation}

The toddler room receives enrollments from two streams: children moving up from the infant classroom, and children enrolled from the external waitlist.

\begin{lemma}[Move-Up Process] \label{lemma:transition}
    \quad
    \begin{enumerate}
        \item [(i)] For the infant room, the average acceptance probability $\alpha_I^*$ is given by Theorem~\ref{thm:fixed_point}. 
        
        \item [(ii)] For the toddler room, (a) a fraction $g/(g + \lambda_T\alpha_T^*)$ of openings are filled by children moving up with acceptance probability one; (b) for the remaining openings, the average acceptance probability among external offers  $\alpha_T^*$ is given by Theorem~\ref{thm:fixed_point}, with expected waiting time
        \begin{equation} \label{eq:wait_T}
            \mathbb{W}_T(\alpha) \;=\; \frac{\alpha\gamma + \delta_T}{\left(\gamma\,(K_T\delta_T - g) - \alpha\gamma\lambda_T - \delta_T \lambda_T\right)^{+}}.
        \end{equation}
        (c) The steady-state fraction of occupied licensed slots is 
        \begin{align}\label{eq:U_T}
            U_T = \frac{g + \lambda_T \alpha_T^*}{K_T \delta_T}.
        \end{align}
        As the move-up rate into the toddler classroom $g$ increases, its average acceptance probability $\alpha_T^*$ strictly decreases and slot utilization $U_T$ strictly increases.
    \end{enumerate}
\end{lemma}

Lemma~\ref{lemma:transition} (i) describes the acceptance probability of an infant classroom of $K_I$ slots, following Theorem~\ref{thm:fixed_point}. 
Lemma~\ref{lemma:transition} (ii) describes the older-age classroom, namely, the toddler classroom in this case. Children moving up occupy $g/\delta_T$ toddler slots on average, so external families face, in effect, a classroom of $K_T - g/\delta_T$ slots. As a result, \eqref{eq:wait_T} is \eqref{eq:alpha_wait} with $K$ replaced by $K_T - g/\delta_T$. The utilization in \eqref{eq:U_T} counts both streams of enrollment.
The move-up process (with rate $g$) improves the toddler slot utilization $U_T$. However, the improvement comes partly at the expense of families with toddlers on the external waitlist, who have to wait longer.


\subsection{Mixed-Age Equilibrium Through Flexible Slots}

We consider a center that converts $f$ toddler slots into \emph{flexible} slots, with the remaining toddler slots and all infant slots staying dedicated. A flexible slot is a slot that is available to admit new infants, infants moving up, and new toddlers. There is a natural upper bound $\hat{f}$ on how many slots can be made flexible, which depends on the rate of infants moving up $g$ and the original number of toddler slots $K_T$. Using Lemma~\ref{lemma:transition}, $\hat{f}$ is the unique solution of
\begin{equation} \label{eq:lendable}
    \hat{f} + \frac{g(\hat{f}) + \lambda_T\,\alpha_T^*(\hat{f})}{\delta_T} = K_T.
\end{equation}
Within this range, the flexible slots are available to infant families whenever the infant classroom needs additional capacity, so it operates as if it had $K_I + f$ slots, while the toddler room retains $K_T$ slots. In what follows, we focus on $f>0$, but the trade-offs apply to all values of $f$. When $f<0$, the only difference is that infant slots are made flexible to admit toddlers. 

Throughout this section, we assume $\lambda_I<K_I\gamma$, $g(0)<K_T\delta_T$ (equivalently, $\hat{f}>0$), and a finite expected wait for external toddlers at $f=\hat{f}$, ensuring finite infant waits and toddler capacity to lend.

\begin{lemma}[Capacity Relief and the Move-Up Flow] \label{lemma:pooling_alpha}
    Let $f \in [0,\hat{f}]$. 
    \begin{enumerate}[label=(\roman*)]
        \item \textbf{(Capacity Relief)} The infant room is characterized by Lemma~\ref{lemma:transition} (i) with $K_I$ replaced by $K_I+f$. As $f$ increases, its average acceptance probability $\alpha_I^*(f)$ strictly increases, expected wait $\mathbb{W}_I(f)$ strictly decreases, and the expected number of occupied infant slots $(K_I+f)U_I(f)$ strictly increases. 
        \item \textbf{(Move-Up Flow)} The toddler room is characterized by Lemma~\ref{lemma:transition} (ii) with $g$ replaced by $g(f)$. If $\sigma>0$, as $f$ increases, its average acceptance probability $\alpha_T^*(f)$ strictly decreases, and the expected number of occupied toddler slots $K_T U_T(f)$ strictly increases. If $\sigma=0$, both are constant. 
    \end{enumerate}
\end{lemma}

Lemma~\ref{lemma:pooling_alpha} (i) pertains to the capacity relief achieved by flexibility. Each flexible slot added to the infant classroom shortens the expected wait for infant enrollments, which raises $\alpha_I^*(f)$. The increased acceptance probability effectively reduces the phantom listings. 

Lemma~\ref{lemma:pooling_alpha} (ii) pertains to the effect of move-up flow to the upper class (namely, the toddler classroom in this case). Slot utilization increases because infants move up and fill slots immediately. As a result, a larger share of slot openings is filled by children moving up, leading to a decrease in external enrollment.

Note that the move-up model adds benefits distinct from pure pooling benefits. Specifically, our policy of making toddler slots flexible for infants ($f>0$) strengthens the move-up effect, increasing the equilibrium enrollment rate. To see this effect, consider the case of moving some infant slot capacity to toddlers ($f<0$). The pooling effects persist, but the move-up effects are weaker. This asymmetry arises from the developmental process unique to early childhood education. 


\subsection{The Optimal Number of Flexible Slots}

Recall the center incurs three sources of costs---wasted-offer cost, idle-capacity cost, and cost of flexibility---as in \S\ref{sec:policies}, and all three change with the number of flexible slots $f$. Absent verification, every listed family is offered a slot exactly once, so class $j$ wastes $\lambda_j(1-\alpha_j^*)$ offers per unit time, each costing $s$. Note that there are no wasted-offer costs associated with infants moving up to older-age classrooms. Empty slots cost $r$ each per unit time. Flexibility carries a staffing-and-layout cost $C(f)$, increasing and strictly convex with $C(0) = 0$. In childcare settings, creating flexible slots requires operational adjustment. For instance, converting a toddler slot to an infant slot requires a crib, diapering and bottle-preparation facilities, and floor space that keeps a non-mobile infant clear of walking toddlers.  In addition, infant slots may require more staff.\footnote{In some states like Michigan, infant and toddler classrooms require the same staff-to-child ratio. In others, infant classrooms have a higher staff-to-child ratio than toddler classrooms. 
Because a mixed group must satisfy the staff-to-child ratio for its youngest child, the separation between age groups may be beneficial to avoid incurring additional staffing costs. 
} 
The center chooses $f$ to minimize
\begin{equation} \label{eq:pooling_objective}
    \mathrm{TC}(f) = \underbrace{s\Bigl[\lambda_I\bigl(1 - \alpha_I^*(f)\bigr) + \lambda_T\bigl(1 - \alpha_T^*(f)\bigr)\Bigr]}_{\text{wasted offers}} + \underbrace{r\left[K_I + K_T - \frac{\lambda_I\alpha_I^*(f)}{\delta_I} - \frac{g(f) + \lambda_T\alpha_T^*(f)}{\delta_T}\right]}_{\text{idle capacity}} + \underbrace{C(f)}_{\text{flexibility}}.
\end{equation}

\begin{proposition}[Optimal Number of Flexible Slots] \label{prop:optimal_pooling}
    $TC(f)$ is strictly convex on $f \in [0,\hat{f}]$, where it has a unique minimizer. A global minimizer $f^*$ of $TC(f)$ exists.
    \begin{enumerate}[label=(\roman*)]
        \item $f^* > 0$ requires the first flexible toddler slot pays for itself, namely, 
        \begin{align}\label{eq:pooling_threshold}
            \lambda_I\,\alpha_I^{*\prime}(0)\,V_I \;+\; \lambda_T\,\alpha_T^{*\prime}(0)\,V_T \;>\; C'(0),
        \end{align}
        where $V_I \;:=\; s \;+\; r/\delta_I \;+\; \sigma\,r/\delta_T$ and $V_T \;:=\; s \;+\; r/\delta_T$ are the values of enrolling one infant and one toddler, respectively. Otherwise, $f^*<0$.\footnote{Negative values of $f$ lend infant slots to the toddler classroom. Infants retain access to all $K_I$ slots, while the toddler classroom behaves as one with $K_T - f$ slots. The infant classroom can lend at most its mean number of idle slots, namely, $K_I - \lambda_I \alpha_I^*(0) / \delta_I$.}
        \item When $f^*>0$, the optimal conversion toward infants $f^*$ is higher when the wasted-offer cost $s$ increases, the idle-slot cost $r$ increases, and the marginal flexibility cost $C'$ decreases. 
    \end{enumerate}
\end{proposition}

Proposition~\ref{prop:optimal_pooling}(i) prescribes the optimal direction of flexibility. Flexibility toward the infant classroom raises the infant room's acceptance rate (by $\lambda_I\alpha_I^{*\prime}$) and yields three key benefits: it (1) reduces wasted offers, saving $s$ for each; (2) accelerates infant enrollment, earning $r/\delta_I$; and (3) accelerates toddler enrollment through the move-up process at rate $\sigma$, earning $\sigma r/\delta_T$. 
On the other hand, if the marginal cost of converting a toddler slot to infants is prohibitively high, then it is better to shift more capacity from the infant classroom to toddlers. 

Flexibility (of either direction) is always valuable.\footnote{Converting the first infant slot into a toddler slot reduces operating costs (e.g., relaxed staff-to-child ratio), as toddler slots are less costly.}
Proposition~\ref{prop:optimal_pooling}(ii) further shows that the optimal conversion toward infants uses more flexible slots when the wasted-offer cost is higher and idle capacity is more costly. Furthermore, if flexibility is less costly, the center chooses more of it.

\paragraph{Numerical Illustration.}

The infant classroom is the one studied in~\S\ref{sec:paradox}: $K_I = 10$ slots facing $\lambda_I = 4.85$ requests per month, with $\gamma = 4$, $p_0 = 0.38$, $\theta = 0.075$, $\delta_I = 1/12$, $s = \$50$, and $r = \$1{,}200$ per slot-month. Following the age bands used by the state of Michigan, under which a child is a toddler from 12 to 30 months of age, a toddler stays for $1/\delta_T = 18$ months on average. Because nearly every infant the center enrolls later needs a toddler place for eighteen months, the toddler room must be larger than the infant room; we consider $K_T = 20$ toddler slots, operated as two groups within Michigan's group-size ceiling. The toddler room draws demand from two sources: external families, who request a slot at rate  $\lambda_T=0.5$ per month, and the center's own infant room, from which a share $\sigma=0.9$ of departing infants move up. We take a flexibility cost of $C(f) = 500f + 500f^2$ per month, which prices the first conversion at the cost of equipping and relicensing one slot for infant care and lets the penalty steepen as flexibility grows. Table~\ref{tab:pooling} summarizes the results, with the optimal configuration highlighted in blue.\footnote{Three flexible slots exceed what the toddler program can lend ($\hat f \approx 2.4$): the toddlers it holds plus three infants would exceed its twenty slots. The row of $f=3$ is reported to show that a third slot would not pay for itself even if it could be lent.}

\begin{table}[ht]
\centering
\caption{Outcomes for various flexible slots $f$. $K_I = 10$, $K_T = 20$, $\lambda_I = 4.85$, $\lambda_T = 0.5$, $\delta_I = 1/12$, $\delta_T = 1/18$, $\gamma = 4$, $p_0 = 0.38$, $\theta = 0.075$, $\sigma = 0.9$, $s = \$50$, $r = \$1{,}200$, $C(f) = 500f + 500f^2$. All costs and frequencies are monthly.}
\label{tab:pooling}
\small
\setlength{\tabcolsep}{3.5pt}
\begin{tabular}{@{}l r r r r r r r r r@{}}
    \toprule
    & & Active & Infant & & Active & Toddler & Slot & Total & \\
    Flexible slots & $\alpha_I^*$ & infant & waitlist & $\alpha_T^*$ & toddler & waitlist & util. & cost & Saving \\
    & & $\lambda_I\alpha_I^*$ & length & & $g + \lambda_T\alpha_T^*$ & length & & & \\
    \midrule
    None ($f = 0$) & 0.15 & 0.72 & 100 & 0.35 & 0.82 & 0.6 & 0.78 & \$7{,}979 & --- \\
    One ($f = 1$) & 0.17 & 0.80 & 83 & 0.34 & 0.89 & 0.8 & 0.86 & \$6{,}338 & 20.6\% \\
    \blue{Two ($f = f^* = 2$)} & \blue{0.18} & \blue{0.88} & \blue{70} & \blue{0.32} & \blue{0.96} & \blue{1.2} & \blue{0.93} & \blue{\$5{,}774} & \blue{27.6\%} \\
    Three ($f = 3$) & 0.20 & 0.96 & 59 & 0.30 & 1.02 & 1.8 & 0.996 & \$6{,}352 & 20.4\% \\
    \bottomrule
\end{tabular}
\end{table}

The optimal policy makes two of the toddler room's twenty slots age-flexible, reducing total monthly cost by $27.6\%$, from \$7{,}979 to \$5{,}774. This configuration raises the infant room's acceptance probability from $15\%$ to $18\%$ and its active demand---the rate at which infants actually enroll---from $0.72$ to $0.88$ children per month. The infant waitlist shrinks from $100$ to $70$ families, and center-wide slot utilization rises from $78\%$ to $93\%$. 
A conventional pooling calculation (that ignores the move-up effect) would indicate one flexible slot as the optimal flexibility level, missing out on additional $7\%$ savings ($27.6\%$ versus $20.6\%$).

\section{Concluding Remarks} \label{sec:conclusion}

Childcare is a central part of the economic infrastructure that supports working families. Yet childcare systems continue to face three closely related challenges: access, affordability, and quality. This paper focuses on access. We model childcare waitlists and propose policies to reduce the informational frictions underlying a widely observed paradox: families face long waitlists even as available childcare capacity goes unused.

We identify two implementable policy improvements. The first, waitlist verification, is a ``high-touch'' service policy that requires centers to balance the staffing cost of contacting families against the costs of wasted offers and idle capacity. Our analysis provides an operational rationale for periodic reconfirmation schedules, such as monthly or quarterly verification, that allow centers to update their waitlists as families' needs and decisions change. The second policy uses mixed-age classrooms to increase flexibility in capacity allocation. Their value extends beyond conventional pooling: centers can use information about children's progression and developmental stages to manage capacity across classrooms more effectively.

Our work also suggests several directions for future research. One concerns centralization. Families frequently apply to multiple providers, and enrollment at one center can leave a family on the waitlists of other centers, creating phantom listings throughout the system. A shared platform or centralized system operated by a public agency could track applications across providers, remove obsolete listings, and coordinate assignment mechanisms, much as centralized enrollment systems do in public education.

A second direction concerns richer enrollment controls. Refundable deposits or renewal fees could encourage families to keep their waitlist status current, although any efficiency gains must be weighed against the possibility of reducing access for families facing liquidity constraints~\citep{zhan2023precharging,hermes2025application}. Centers could also consider making multiple offers for a single vacancy, trading faster placement against the risk that more than one family accepts~\citep{epstein2024selection}. More generally, offers could be coordinated with predictable departures, age transitions, and internal move-ups. Because admission dates and birth months affect classroom tenure and the timing of future vacancies, jointly scheduling transitions and offers may substantially reduce idle capacity~\citep{ofarrelly2013transitions}.

Our model abstracts from pricing, but the design and timing of public subsidies and employer-sponsored family benefits can affect both families' application decisions and centers' enrollment outcomes. An important question is whether financial support should be shifted toward the childcare years, when families often face particularly acute financial and employment constraints, rather than concentrated later in children's education~\citep{ho2020efficient,arora2021improving,costasfernandez2026familyfriendly}. More broadly, incorporating prices, subsidies, and liquidity constraints into models of childcare access could reveal important interactions between financial support and capacity utilization.

Finally, there is often an interaction between access and quality of care. 
Increasing capacity utilization is valuable only if it can be achieved without compromising the quality and continuity of care. Future work could examine how staffing levels, caregiver assignments, caregiver continuity, classroom composition, and children's developmental outcomes interact with admission and capacity decisions~\citep{hotz2011,le2015caregiving,choi2018continuity}. This would allow operational policies aimed at improving access to be evaluated not only by how quickly children are placed, but also by the quality of care they receive.

These questions point to a broader operations research agenda for childcare. Many of the underlying problems involve the allocation of scarce capacity under uncertain and changing demand, yet the institutional details of childcare create features that differ from those of conventional service systems. To help ground this research in families' experiences, we maintain a public database of survey responses from working families in academia about their childcare experiences at \url{https://childcare-om.vercel.app/}. 

We hope this paper takes a step toward a more systematic understanding of childcare operations. Childcare offers fertile ground for operations research and operations management, with questions spanning queueing and matching, capacity and workforce design, information disclosure, contracting, and field experimentation. It also provides an opportunity to translate operations research into tangible social impact by helping public and private investment produce more accessible, affordable, and high-quality care for children while improving the economic security and well-being of families.


\bibliography{bibliography}
\appendix

\newpage
\section{Proofs from Section~\ref{sec:paradox}} 

\subsection{Proof of Lemma~\ref{lemma:paradox}}

Throughout the proof, $p(w)=p$, so every offer is accepted with probability $p$, irrespective of the recipient's elapsed wait, and distinct offers are resolved independently. For a given slot, with a family always available to contact, an enrollment epoch regenerates the slot. We use the renewal-reward theorem to evaluate its long-run averages as ratios of expectations over a single i.i.d.\ cycle. 

Fix a slot and let a cycle begin at an enrollment. The slot is occupied for a mean $1/\delta$ time units and then stands open until the next enrollment. While it is open, the center keeps exactly one offer outstanding. Each offer is resolved after a mean $1/\gamma$ time units and is accepted with probability $p$, so the number of offers made during the open stage is geometric with mean $1/p$. The open stage therefore lasts a mean $1/(p\gamma)$ time units and removes a mean $1/p$ families from the waitlist, exactly one of whom enrolls. Successive cycles are independent and identically distributed, with finite mean length $1/\delta+1/(p\gamma)$. By the renewal-reward theorem, rewarding the open stage gives the long-run fraction of time that a slot stands open: 
\begin{align*}
    \frac{1/(p\gamma)}{1/\delta+1/(p\gamma)}=\frac{\delta}{p\gamma+\delta},
\end{align*}
which is~\eqref{eq:open_share}. Rewarding each family removed gives the long-run rate at which one slot clears the waitlist: 
\begin{align*}
    \frac{1/p}{1/\delta+1/(p\gamma)}=\frac{\gamma\delta}{p\gamma+\delta}.
\end{align*}
Slots become available independently of one another and each runs its own offer process, so summing the above display over the $K$ slots gives the effective clearance rate
\begin{align*}
    \mu_{\mathrm{eff}}=\frac{K\gamma\delta}{p\gamma+\delta},
\end{align*}
as in~\eqref{eq:clearing_rate}, of which a fraction $p$ are enrollments. Because this calculation assumes that a family is available to contact whenever the slot is open, so $\mu_{\mathrm{eff}}$ is the largest rate at which the center can clear its waitlist.

\noindent\emph{(i)} We consider two cases. 

\textbf{Case 1:} $\lambda<\mu_{\mathrm{eff}}$. The waitlist is stable, so its long-run clearance rate equals the arrival rate $\lambda$. Hence, every listed family is eventually made an offer. Each cleared family accepts with probability $p$, giving an enrollment rate of $\lambda p$. Each enrollment occupies one slot for a mean $1/\delta$ time units. Little's law applied to the occupied slots therefore gives a mean of $\lambda p/\delta$ occupied slots and
\begin{align*}
    U=\frac{\lambda p}{K\delta}.
\end{align*}

\textbf{Case 2:} $\lambda\geq\mu_{\mathrm{eff}}$. Requests arrive at least as fast as the center can remove families, so the long-run fraction of time with an empty waitlist is zero and the cycle structure above applies unmodified. Rewarding the occupied stage, the renewal-reward theorem gives
\begin{align*}
    U=\frac{1/\delta}{1/\delta+1/(p\gamma)}=\frac{p\gamma}{p\gamma+\delta}.
\end{align*}

Since
\begin{align*}
    \lambda<\mu_{\mathrm{eff}}
    \quad\Longleftrightarrow\quad
    \frac{\lambda p}{K\delta}<\frac{p\gamma}{p\gamma+\delta},
\end{align*}
the applicable branch in each case is the minimum in~\eqref{eq:umax}, and the two branches agree at $\lambda=\mu_{\mathrm{eff}}$. Finally,
\begin{align*}
    U\leq\frac{p\gamma}{p\gamma+\delta}=1-\frac{\delta}{p\gamma+\delta}<1,
\end{align*}
where the last inequality follows from $\delta>0$. 

\noindent\emph{(ii)}  Based on our model introduced in~\S\ref{subsec:primitives}, we assume that the time from joining the waitlist to receiving an offer is an exponential random variable with rate $\mu_{\mathrm{eff}}-\lambda$ when $\lambda<\mu_{\mathrm{eff}}$; otherwise, its mean is infinity. 

\textbf{Case 1:} $\lambda<\mu_{\mathrm{eff}}$. The exponentially distributed $w$ with rate $\mu_{\mathrm{eff}}-\lambda$ has a finite mean of 
\begin{align*}
    \mathbb{W}=\frac{1}{\mu_{\mathrm{eff}}-\lambda}
    =\frac{1}{\frac{K\gamma\delta}{p\gamma+\delta}-\lambda}
    =\frac{p\gamma+\delta}{K\gamma\delta-p\gamma\lambda-\delta\lambda}.
\end{align*}

\textbf{Case 2:} $\lambda\geq\mu_{\mathrm{eff}}$. We set $\mathbb{W}=\infty$.

Both cases give~\eqref{eq:queue_length}, with Case 2 encoded by the positive part in~\eqref{eq:queue_length}. 

Finally, with constant acceptance probability $p$, the expected active share of the waitlist is $p$.
\qed

\subsection{Proof of Proposition~\ref{prop:paradox}}

Equation~\eqref{eq:umax} and the identity $1-\min\{a,b\}=\max\{1-a,1-b\}$ give
\begin{align*}
\begin{aligned}
    K(1-U)
    &=\max\left\{K-\frac{\lambda p}{\delta},\;K\left(1-\frac{p\gamma}{p\gamma+\delta}\right)\right\}\\
    &=\max\left\{K-\frac{\lambda p}{\delta},\;\frac{K\delta}{p\gamma+\delta}\right\}
    \geq\frac{K\delta}{p\gamma+\delta}>0,
\end{aligned}
\end{align*}
where the last inequality follows from $K\geq1$, $\delta>0$, and $p\gamma+\delta>0$.

\textbf{Case 1:} $\lambda<\mu_{\mathrm{eff}}$. Recalling $\mu_{\mathrm{eff}} = K\gamma\delta / (p\gamma +\delta)$, Lemma~\ref{lemma:paradox}(ii) gives
\begin{align*}
    \mathbb{W} = \frac{1}{\mu_{\mathrm{eff}}-\lambda} > \frac{1}{\mu_{\mathrm{eff}}} > 0,
\end{align*}
where the first inequality is strict because $\lambda>0$. 

\textbf{Case 2:} $\lambda\geq\mu_{\mathrm{eff}}$. Lemma~\ref{lemma:paradox}(ii) gives $\mathbb{W}=\infty$ and the bound holds trivially.

Below we prove the second half of the statement. When $\lambda<\mu_{\mathrm{eff}}$, $\mathbb{W}<\infty$. Note that the effective clearance rate $\mu_{\mathrm{eff}}$ does not depend on $\lambda$, so
\begin{align*}
    \frac{\partial\mathbb{W}}{\partial\lambda} = \frac{1}{(\mu_{\mathrm{eff}}-\lambda)^2} > 0,
\end{align*}
implying that $\mathbb{W}$ monotonically increases in $\lambda$. In particular, $\mathbb{W}\to\infty$ as $\lambda\uparrow\mu_{\mathrm{eff}}$, after which it remains infinite, because $\mathbb{W} = \infty$ when $\lambda\geq\mu_{\mathrm{eff}}$. 

From Lemma~\ref{lemma:paradox}(i), the number of empty slots $K(1-U)$ is given by 
\begin{align*}
    \max\left\{K - \frac{\lambda p}{\delta}, \frac{K\delta}{p\gamma+\delta}\right\}.
\end{align*}
Note that $K-\lambda p/\delta$ is strictly decreasing in $\lambda$ and $K\delta/(p\gamma+\delta)$ does not depend on $\lambda$, so their maximum never falls below $K\delta/(p\gamma+\delta)$.
\qed

\subsection{Proof of Lemma~\ref{lemma:comp_statics}}

From Lemma~\ref{lemma:paradox}(ii), 
the condition $\mathbb{W}<\infty$ is equivalent to $\lambda<\mu_{\mathrm{eff}}$. 
When $\lambda<\mu_{\mathrm{eff}}$, the first branch of~\eqref{eq:umax}, from Lemma~\ref{lemma:paradox}(i), applies, so $U=\lambda p/(K\delta)$ and $\partial U/\partial\gamma=0$.
\qed

\subsection{Proof of Lemma~\ref{lemma:alpha-u-wait}}

Throughout this proof, $\alpha$ denotes the long-run fraction of resolved offers that are accepted, taken as given. Unlike in Lemma~\ref{lemma:paradox}, acceptance probabilities depend on elapsed waits, and successive slot cycles need not be independent. We therefore use long-run offer counts rather than an i.i.d.\ renewal-cycle argument. 

Enumerate resolved offers across the center, and let $W_m$ be the recipient's elapsed wait at contact and $Y_m$ the acceptance indicator for offer $m$. Under the model's  conditional acceptance specification,  
$Y_m - p_0 \cdot e^{-\theta W_m}$ has conditional mean zero given the history immediately before its outcome is revealed. These differences are bounded, so the martingale strong law gives
\begin{align*}
    \frac{1}{n}\sum_{m=1}^{n}Y_m - \frac{1}{n}\sum_{m=1}^{n}p_0 \cdot e^{-\theta W_m}  \to 0,  \quad\text{almost surely}.
\end{align*}
Since the first average converges to $\alpha$, the average acceptance frequency equals the limiting average acceptance probability, as in~\eqref{eq:alpha_def}. For $\alpha > 0$, the corresponding long-run numbers of offers and wasted offers per enrollment are $1/\alpha$ and $(1-\alpha)/\alpha$, respectively. 

We first calculate the effective clearance rate, assuming a family is always available whenever a slot can issue an offer. Let $N(t)$ be the number of offers resolved by time $t$, and let $H(t)$ be the total occupied-slot time accumulated across the $K$ slots. Under the aforementioned assumption, every unoccupied slot holds an outstanding offer. Offer resolutions therefore have cumulative intensity $\gamma (Kt - H(t))$, while departures have cumulative intensity $\delta H(t)$. The counting-process martingale strong law gives errors of order $o(t)$ almost surely. Moreover, accepted offers and departures differ by at most $K$, and the number of accepted offers is $\alpha N(t) + o(N(t))$. Since the offer-resolution intensity is bounded by $K\gamma$, $N(t) = O(t)$ almost surely. Consequently, 
\begin{align*}
    N(t) = \gamma (Kt - H(t)) + o(t),
    \qquad 
    \alpha N(t) = \delta H(t) + o(t).
\end{align*}
Solving these identities yields
\begin{align*}
    \frac{N(t)}{t} \to \frac{K\gamma\delta}{\alpha\gamma+\delta}, 
    \qquad 
    \frac{H(t)}{Kt} \to \frac{\alpha \gamma}{\alpha\gamma + \delta}. 
\end{align*}
Thus, the long-run fraction of time that a slot stands open is $\delta/(\alpha\gamma+\delta)$, reproducing~\eqref{eq:open_share} with $p$ replaced by $\alpha$, and the effective clearance rate is 
\begin{align*}
    \mu_{\mathrm{eff}}(\alpha) =\frac{K\gamma\delta}{\alpha\gamma+\delta}.
\end{align*}    
This establishes~\eqref{eq:clearing_rate} with $p$ replaced by $\alpha$, without assuming that slots have identical realized acceptance frequencies.

\noindent\textit{(i)} We first derive the expression for $U(\alpha)$. We consider two cases. 

\textbf{Case 1:} $\lambda<\mu_{\mathrm{eff}}(\alpha)$. The waitlist is positive recurrent, so flow balance gives an offer-resolution rate $\lambda$. A fraction $\alpha$ of resolved offers are accepted, yielding an enrollment rate of $\lambda\alpha$. Since each enrollment occupies one slot for a mean $1/\delta$ time units, Little's law applied to the occupied slots gives a mean of $\lambda\alpha/\delta$ occupied slots and therefore
\begin{align*}
    U(\alpha)=\frac{\lambda\alpha}{K\delta}.
\end{align*}

\textbf{Case 2:} $\lambda\geq\mu_{\mathrm{eff}}(\alpha)$. Requests then arrive at least as fast as the center removes families, so the long-run fraction of time with an empty waitlist is zero. From the analysis above, the long-run fraction of time that the slot stands occupied gives 
\begin{align*}
    U(\alpha)=\frac{\alpha\gamma}{\alpha\gamma+\delta}.
\end{align*}

For $\alpha>0$, 
\begin{align*}
    \lambda<\mu_{\mathrm{eff}}(\alpha)
    \quad\Longleftrightarrow\quad
    \frac{\lambda\alpha}{K\delta}<\frac{\alpha\gamma}{\alpha\gamma+\delta}. 
\end{align*}
Thus, the applicable branch in each case is the minimum in~\eqref{eq:alpha_u}, and the two branches agree at $\lambda=\mu_{\mathrm{eff}}(\alpha)$. For $\alpha=0$, both branches vanish, consistently with the aggregate accounting above. Here, $\alpha=0$ means that the acceptance fraction vanishes; it does not require that no offers are ever accepted.

\noindent\textit{(ii)} We next derive the expression for $\mathbb{W}(\alpha)$.  As in the proof of Lemma~\ref{lemma:paradox}(ii), based on the model introduced in~\S\ref{subsec:impatience},  the time from joining the waitlist to receiving an offer is assumed to be exponential with rate $\mu_{\mathrm{eff}}(\alpha)-\lambda$ when $\lambda<\mu_{\mathrm{eff}}(\alpha)$, and with infinite mean otherwise.

\textbf{Case 1:} $\lambda<\mu_{\mathrm{eff}}(\alpha)$. The exponentially distributed $w$ with rate $\mu_{\mathrm{eff}}(\alpha)-\lambda$ has a finite mean of 
\begin{align*}
    \mathbb{W}=\frac{1}{\mu_{\mathrm{eff}(\alpha)}-\lambda}
    =\frac{1}{\frac{K\gamma\delta}{\alpha\gamma+\delta}-\lambda}
    =\frac{\alpha\gamma+\delta}{K\gamma\delta-\alpha\gamma\lambda-\delta\lambda}.
\end{align*}

\textbf{Case 2:} $\lambda\geq\mu_{\mathrm{eff}}(\alpha)$. We set $\mathbb{W}=\infty$.

Both cases give~\eqref{eq:alpha_wait}, with Case 2 encoded by the positive part in~\eqref{eq:alpha_wait}. 
\qed

\subsection{Proof of Theorem~\ref{thm:fixed_point}}

Let $\Phi$ denote the right-hand side of~\eqref{eq:alpha_def}. When $\lambda < \mu_{\mathrm{eff}}(\alpha)$, the model assumes that the waiting time of a family from joining the waitlist to receiving an offer is exponentially distributed with rate $\mu_{\mathrm{eff}}(\alpha) - \lambda$, and therefore, 
\begin{align*}
    \Phi(\alpha) = p_0 \int_{0}^{\infty} e^{-\theta w} (\mu_{\mathrm{eff}}(\alpha) - \lambda) e^{-(\mu_{\mathrm{eff}}(\alpha) - \lambda)w} = \frac{p_0\bigl(\mu_{\mathrm{eff}}(\alpha)-\lambda\bigr)} {\mu_{\mathrm{eff}}(\alpha)-\lambda+\theta}.
\end{align*}
When $\lambda\geq\mu_{\mathrm{eff}}(\alpha)$, waiting time explodes, so we extend $\Phi$ by setting $\Phi(\alpha)=0$. Thus, using $\mathbb{W}(\alpha) = 1/(\mu_{\mathrm{eff}}(\alpha)-\lambda)^+$ from~\eqref{eq:alpha_wait}, $\Phi$ can be equivalently written as 
\begin{align*}
    \Phi(\alpha) =\frac{p_0(\mu_{\mathrm{eff}}(\alpha)-\lambda)^{+}}{(\mu_{\mathrm{eff}}(\alpha)-\lambda)^{+} + \theta} =\frac{p_0}{1+\theta\mathbb{W}(\alpha)},
\end{align*}
where the last expression is interpreted as zero when $\mathbb{W}(\alpha)=\infty$. This establishes~\eqref{eq:alpha_fp}. 

Since
\begin{align*}
    \mu_{\mathrm{eff}}(\alpha) =\frac{K\gamma\delta}{\alpha\gamma+\delta}
\end{align*}
is continuous and strictly decreasing in $\alpha$, the expression for $\Phi$ as a function of $\mu_{\mathrm{eff}}$ implies that $\Phi$ is continuous and nonincreasing, including at the stability boundary. Moreover, $\theta>0$ implies
\begin{align*}
    0 \leq \Phi(\alpha) < p_0.
\end{align*}

Consequently, $\alpha-\Phi(\alpha)$ is continuous and strictly increasing on $[0,p_0]$, with
\begin{align*}
    0-\Phi(0)\leq0, 
    \qquad
    p_0-\Phi(p_0)>0.
\end{align*}
The intermediate value theorem gives a unique zero $\alpha^*\in[0,p_0)$. Finally, the identity $\Phi(\alpha)=p_0/(1+\theta\mathbb{W}(\alpha))$ shows that~\eqref{eq:alpha_def} and~\eqref{eq:alpha_fp} have the same solutions. Hence $\alpha^*$ is also the unique solution of~\eqref{eq:alpha_fp}. 
\qed

\subsection{Auxiliary Results}

The following auxiliary result will be used in multiple subsequent proofs.

\begin{lemma}\label{lemma:alpha-star}
\quad
\begin{enumerate}
    \item [(i)] If $\lambda<K\gamma$, then
    \begin{align*}
        0<\alpha^*<\bar{\alpha} :=\frac{\delta(K\gamma-\lambda)}{\lambda\gamma},
        \qquad
        \lambda<\mu_{\mathrm{eff}}(\alpha^*),
        \qquad
        \mathbb{W}(\alpha^*)<\infty.
    \end{align*}

    \item [(ii)] If $\lambda\geq K\gamma$, then
    \begin{align*}
        \alpha^*=0,
        \qquad
        \lambda\geq\mu_{\mathrm{eff}}(\alpha^*),
        \qquad
        \mathbb{W}(\alpha^*)=\infty.
    \end{align*}
\end{enumerate}
\end{lemma}

\noindent\textit{Proof of Lemma~\ref{lemma:alpha-star}.}
Let $\Phi$ denote the right-hand side of~\eqref{eq:alpha_def}, or equivalently, the right-hand side of~\eqref{eq:alpha_fp}. We know that $\mathbb{W}(\alpha) < \infty$ if $\lambda<\mu_{\mathrm{eff}}(\alpha)$, and $\mathbb{W} = \infty$ if $\lambda \geq \mu_{\mathrm{eff}}(\alpha)$. Since $\Phi(\alpha) = p_0/(1+\theta\mathbb{W}(\alpha))$ with $p_0 > 0$ and $\theta > 0$, it follows that 
\begin{align}\label{eq:Phi-two-cases}
    \Phi(\alpha)>0
    \ \mbox{ if } \
    \lambda<\mu_{\mathrm{eff}}(\alpha), 
    \quad
    \Phi(\alpha)=0
    \ \mbox{ if } \
    \lambda\geq\mu_{\mathrm{eff}}(\alpha). 
\end{align}

\textit{(i)} $\lambda<K\gamma$. Since $\mu_{\mathrm{eff}}(0)=K\gamma>\lambda$, we have $\Phi(0)>0$, so zero cannot be a fixed point and thus $\alpha^*>0$. The fixed-point equation $\Phi(\alpha^*) = \alpha^*$ then gives $\Phi(\alpha^*)>0$, which, by~\eqref{eq:Phi-two-cases}, implies
\begin{align*}
    \lambda<\mu_{\mathrm{eff}}(\alpha^*)
    =\frac{K\gamma\delta}{\alpha^*\gamma+\delta}.
\end{align*}
Rearranging gives 
\begin{align*}
    \alpha^* < \frac{\delta(K\gamma-\lambda)}{\lambda\gamma} =: \bar{\alpha},
\end{align*}
and Lemma~\ref{lemma:alpha-u-wait}(ii) yields $\mathbb{W}(\alpha^*)<\infty$.

\textit{(ii)} $\lambda\geq K\gamma$. For every $\alpha\geq0$,
\begin{align*}
    \mu_{\mathrm{eff}}(\alpha)
    =\frac{K\gamma\delta}{\alpha\gamma+\delta}
    \leq K\gamma\leq\lambda.
\end{align*}
Hence, by~\eqref{eq:Phi-two-cases}, $\Phi\equiv0$. The fixed-point equation $\Phi(\alpha^*) = \alpha^*$ then gives $\alpha^*=0$. Consequently, $\mu_{\mathrm{eff}}(\alpha^*)=\mu_{\mathrm{eff}}(0) = K\gamma\leq\lambda$ and $\mathbb{W}(\alpha^*)=\infty$.
\qed

\subsection{Proof of Proposition~\ref{prop:amplified_paradox}}

From Theorem~\ref{thm:fixed_point}, $\alpha^*<p_0$. Moreover, note that both functions inside the minimum in~\eqref{eq:alpha_u} are strictly increasing in $\alpha$, so their minimum, namely, the slot utilization, is strictly increasing in $\alpha$. Hence, $\alpha^* < p_0$ implies $U(\alpha^*)<U(p_0)$. 

We now compare the expected waiting times by considering three cases.

\textbf{Case 1:} $\lambda<\mu_{\mathrm{eff}}(p_0)$. Both $\alpha^*$ and $p_0$ lie in the finite-wait regime, where differentiating~\eqref{eq:alpha_wait} gives
\begin{align*}
    \frac{\partial\mathbb{W}(\alpha)}{\partial\alpha} 
    =\frac{K\gamma^2\delta}{\left(K\gamma\delta-\lambda(\alpha\gamma+\delta)\right)^2}>0,
\end{align*}
where the denominator is positive because $\mu_{\mathrm{eff}}(\alpha) > \mu_{\mathrm{eff}}(p_0) > \lambda$. Hence, $\alpha^*<p_0$ implies $\mathbb{W}(\alpha^*)<\mathbb{W}(p_0)$.

\textbf{Case 2:} $\mu_{\mathrm{eff}}(p_0)\leq\lambda<K\gamma$. Equation~\eqref{eq:alpha_wait} gives $\mathbb{W}(p_0)=\infty$, whereas Lemma~\ref{lemma:alpha-star} gives $\mathbb{W}(\alpha^*)<\infty$. Hence, $\mathbb{W}(\alpha^*)<\mathbb{W}(p_0)$.

\textbf{Case 3:} $\lambda\geq K\gamma$. Lemma~\ref{lemma:alpha-star} gives $\alpha^*=0$ and $\mathbb{W}(\alpha^*)=\infty$. Moreover, since $\mu_{\mathrm{eff}}(p_0)<K\gamma\leq\lambda$, the constant-$p_0$ waitlist is unstable as well, so $\mathbb{W}(p_0)=\infty$. Therefore, $\mathbb{W}(p_0)=\mathbb{W}(\alpha^*)=\infty$. 
\qed

\subsection{Proof of Lemma~\ref{lemma:alpha_cs}}

For $x\in\{\lambda,\theta,\gamma\}$, write $\Phi(\alpha;x)$ for the right-hand side of~\eqref{eq:alpha_def} and $\alpha^*(x)$ for its unique fixed point. By the proof of Theorem~\ref{thm:fixed_point}, 
\begin{align*}
    \Phi(\alpha;x) = p_0\frac{(\mu_{\mathrm{eff}}(\alpha)-\lambda)^{+}}{(\mu_{\mathrm{eff}}(\alpha)-\lambda)^{+}+\theta},
    \qquad
    \mu_{\mathrm{eff}}(\alpha) =\frac{K\gamma\delta}{\alpha\gamma+\delta}.
\end{align*}


\noindent\textit{(i)} Monotonicity of $\alpha^*$.
At fixed $\alpha$, the preceding expression shows that $\Phi$ is nonincreasing in $\lambda$ and $\theta$. It is nondecreasing in $\gamma$, since
\begin{align*}
    \frac{\partial\mu_{\mathrm{eff}}(\alpha)}{\partial\gamma} = \frac{K\delta^2}{(\alpha\gamma+\delta)^2}>0.
\end{align*}

To transfer these comparisons to the fixed point, recall that $\alpha-\Phi(\alpha;x)$ is strictly increasing in $\alpha$, with unique zero $\alpha^*(x)$. If $\Phi(\alpha;x')\leq\Phi(\alpha;x)$ for every $\alpha$, then
\begin{align*}
    \alpha^*(x)-\Phi(\alpha^*(x);x') \geq \alpha^*(x)-\Phi(\alpha^*(x);x) =0,
\end{align*}
and hence $\alpha^*(x')\leq\alpha^*(x)$. Reversing the inequalities gives the corresponding comparison when $\Phi$ increases. Thus, $\alpha^*$ is nonincreasing in $\lambda$ and $\theta$ and nondecreasing in $\gamma$. Wherever the derivatives exist,
\begin{align*}
    \frac{\partial\alpha^*}{\partial\lambda}\leq0,
    \qquad
    \frac{\partial\alpha^*}{\partial\theta}\leq0,
    \qquad
    \frac{\partial\alpha^*}{\partial\gamma}\geq0.
\end{align*}

\noindent\textit{(ii)} Monotonicity of waiting time and utilization.
We consider two cases. 

\textbf{Case 1:} $\lambda<K\gamma$. Lemma~\ref{lemma:alpha-star} gives $\alpha^*>0$ and $\mathbb{W}(\alpha^*)<\infty$. The fixed-point identity~\eqref{eq:alpha_fp} and the utilization formula~\eqref{eq:alpha_u} then yield
\begin{align*}
    \mathbb{W}(\alpha^*)
    =\frac{1}{\theta}\left(\frac{p_0}{\alpha^*}-1\right),
    \qquad
    U(\alpha^*)=\frac{\lambda\alpha^*}{K\delta}.
\end{align*}
Since $\alpha^*$ is nondecreasing in $\gamma$ by part~(i), the expected waiting time is nonincreasing and utilization is nondecreasing in $\gamma$.

\textbf{Case 2:} $\lambda\geq K\gamma$. Lemma~\ref{lemma:alpha-star} and~\eqref{eq:alpha_u} give
\begin{align*}
    \alpha^*=0,\qquad
    \mathbb{W}(\alpha^*)=\infty,\qquad
    U(\alpha^*)=0.
\end{align*}
These values extend both monotonicity conclusions across $\gamma=\lambda/K$, with waiting time interpreted as an extended-real-valued function. Consequently,
\begin{align*}
    \frac{\partial\mathbb{W}(\alpha^*)}{\partial\gamma}\leq0,
    \qquad
    \frac{\partial U(\alpha^*)}{\partial\gamma}\geq0
\end{align*}
where the respective quantities are finite and differentiable.
\qed

\section{Proofs from Section~\ref{sec:reverify}}

\subsection{Proof of Theorem~\ref{thm:alpha_nu}}

For simplicity, write $x:=\phi\nu$. 

Active and phantom families join at rates $\lambda p_0$ and $\lambda(1-p_0)$, respectively, and each active family becomes phantom at rate $\theta$. Every listed family receives an offer at rate $1/\mathbb{W}(\nu)$, and verification removes each phantom at rate $x$. Thus, the active and phantom flow balances read
\begin{align*}
    \lambda p_0=\left(\theta+\frac{1}{\mathbb{W}(\nu)}\right)Q_a(\nu),
    \qquad
    \lambda(1-p_0)+\theta Q_a(\nu)=\left(x+\frac{1}{\mathbb{W}(\nu)}\right)Q_p(\nu).
\end{align*}
Solving the above two equations yields
\begin{align*}
    Q_a(\nu)=\frac{\lambda p_0 \mathbb{W}(\nu)}{1+\theta \mathbb{W}(\nu)},
    \qquad
    Q_p(\nu)=\frac{\lambda \mathbb{W}(\nu)(1-p_0+\theta \mathbb{W}(\nu))}{(1+\theta \mathbb{W}(\nu))(1+x \mathbb{W}(\nu))}.
\end{align*}
Note that capacity constraint reads
\begin{align}\label{eq:capacity_const}
    \left(\frac{Q_a(\nu)}{\delta} + \frac{Q_a(\nu) + Q_p(\nu)}{\gamma} \right) \frac{1}{\mathbb{W}} = K.
\end{align}
Substituting for $Q_a(\nu)$ and $Q_p(\nu)$ using their expressions derived in the above then implies 
\begin{align*}
    \frac{\lambda p_0}{1+\theta \mathbb{W}(\nu)} \left(\frac{1}{\delta} + \frac{1}{\gamma}\right) + \frac{\lambda (1-p_0+\theta \mathbb{W}(\nu))}{(1+\theta \mathbb{W}(\nu))(1+x\mathbb{W}(\nu))} \cdot \frac{1}{\gamma} = K, 
\end{align*}
which establishes~\eqref{eq:capacity_const-alt}.

Write
\begin{align*}
    \Psi(w,x) := \frac{\lambda p_0}{1+\theta w} \left(\frac{1}{\delta} + \frac{1}{\gamma}\right) + \frac{\lambda (1-p_0+\theta w)}{(1+\theta w)(1+xw)} \cdot \frac{1}{\gamma} - K.
\end{align*}
Then, $\Psi(\mathbb{W}(\nu),x) = 0$. 
Differentiating $\Psi(w,x)$ with respect to $w$ yields
\begin{align*}
    \frac{\partial \Psi(w,x)}{\partial w} = -\frac{\lambda p_0\theta}{\delta(1+\theta w)^2} - \frac{\lambda p_0\theta xw}{\gamma(1+\theta w)^2(1+xw)} - \frac{\lambda x(1-p_0+\theta w)}{\gamma(1+\theta w)(1+xw)^2} < 0.
\end{align*}
Thus, $\Psi(w,x)$ is strictly decreasing in $w$. 
Evaluating $\Psi$ at boundary values, for $x \geq 0$, 
\begin{align*}
    \Psi(0+,x) = \lambda p_0 \left(\frac{1}{\delta} + \frac{1}{\gamma}\right) + \lambda (1-p_0) \cdot \frac{1}{\gamma} - K = \frac{\lambda p_0}{\delta} + \frac{\lambda}{\gamma} - K > 0,
\end{align*}
where the strict inequality follows from $\lambda p_0 > K\gamma\delta/(\gamma+\delta)$ and $\lambda \geq \lambda p_0$.  When $x=0$, $\Psi(\infty,0)=\lambda/\gamma-K<0$; when $x>0$,
$\Psi(\infty,x)=-K<0$. Hence, for any given $x$, there exists a unique positive root of $\Psi(w,x)=0$, that is, $\mathbb{W}(\nu)$. 
Additionally, for $w > 0$, 
\begin{align*}
     \frac{\partial \Psi(w,x)}{\partial x} = -\frac{\lambda (1-p_0+\theta w)w}{(1+\theta w)(1+xw)^2} \cdot \frac{1}{\gamma} < 0. 
\end{align*}
Then, the implicit-function theorem implies 
\begin{align*}
    \frac{d\mathbb{W}(\nu)}{d\nu} = -\phi\,\frac{\partial \Psi(w,x)/\partial x}{\partial \Psi(w,x)/\partial w}<0.    
\end{align*}
implying that $\mathbb{W}(\nu)$ is strictly decreasing in $\nu$.

Because offers reach all listed families at the same rate and are accepted precisely by active families, the acceptance probability equals the active share:
\begin{align*}
    \alpha^*(\nu)
    =\frac{Q_a(\nu)}{Q_a(\nu)+Q_p(\nu)}
    =\frac{p_0(1+x\mathbb{W}(\nu))}{1+(\theta+p_0x)\mathbb{W}(\nu)},
\end{align*}
which establishes~\eqref{eq:alpha_nu}. 
Since $\mathbb{W}(\nu)$ strictly decreases in $\nu$, we have
\begin{align*}
    \frac{Q_a(\nu)}{\mathbb{W}(\nu)} = \frac{\lambda p_0}{1+\theta\mathbb{W}(\nu)}
\end{align*}
strictly increases. The capacity balance implies that
\begin{align*}
    \frac{Q_a(\nu)+Q_p(\nu)}{\mathbb{W}(\nu)} = K\gamma-\frac{\gamma}{\delta}\frac{\lambda p_0}{1+\theta\mathbb{W}(\nu)},
\end{align*}
which strictly decreases. Their ratio, $\alpha^*(\nu)=Q_a(\nu)/(Q_a(\nu)+Q_p(\nu))$, therefore strictly increases in $\nu$.

At $\nu=0$, the total flow balance gives
\begin{align*}
    Q_a(0)+Q_p(0)=\lambda\mathbb{W}(0).
\end{align*}
The capacity balance therefore implies
\begin{align*}
    \frac{Q_a(0)}{\mathbb{W}(0)}=\delta\left(K-\frac{\lambda}{\gamma}\right).
\end{align*}
Consequently,
\begin{align*}
    \alpha^*(0)=\frac{Q_a(0)}{Q_a(0)+Q_p(0)}=\frac{\delta(K\gamma-\lambda)}{\lambda\gamma}.
\end{align*}
Using $Q_a(0)/\mathbb{W}(0)=\lambda p_0/(1+\theta\mathbb{W}(0))$ yields
\begin{align*}
    \mathbb{W}(0) = \frac{1}{\theta}\left(\frac{\lambda p_0}{\delta(K-\lambda/\gamma)}-1\right) < \infty. 
\end{align*}

Because $Q_p(\nu)>0$ for every finite $\nu$, the capacity
balance implies
\begin{align*}
    \frac{\lambda p_0}{1+\theta\mathbb{W}(\nu)}\left(\frac{1}{\delta}+\frac{1}{\gamma}\right)<K.    
\end{align*}
Hence, 
\begin{align*}
    \mathbb{W}(\nu)>\frac{1}{\theta}\left(\frac{\lambda p_0(\gamma+\delta)}{K\gamma\delta}-1\right)>0.
\end{align*}
Also, the expression for $Q_p(\nu)$ from~\eqref{eq:AP_nu} implies $0\leq Q_p(\nu)\leq\lambda/(\phi\nu)$ for $\nu>0$. Thus, $Q_p(\nu)\to0$, as $\nu \to \infty$, whereas the preceding lower bound on $\mathbb{W}(\nu)$ ensures that $Q_a(\nu)$ remains bounded away from zero, recalling the expression for $Q_a(\nu)$ from~\eqref{eq:AP_nu}. With finite $Q_a(\nu)$, we have $\lim_{\nu\to\infty} \alpha^*(\nu) = \lim_{\nu\to\infty} Q_a(\nu) / (Q_a(\nu) + Q_p(\nu)) = 1$. Substituting for $x \to \infty$ in~\eqref{eq:capacity_const-alt} yields 
\begin{align*}
    \lim_{\nu\to\infty}\mathbb{W}(\nu) = \frac{1}{\theta} \left(\frac{\lambda p_0(\gamma+\delta)}{K\gamma\delta}-1\right) > 0.
\end{align*}
Recalling that $\mathbb{W}(\nu)$ strictly decreases in $\nu$, it follow that $\mathbb{W}(\nu) \in (0,\infty)$ for all $\nu \geq 0$. 
\qed

\subsection{Proof of Lemma~\ref{lemma:reverify_composition}}

Recall from the proof of Theorem~\ref{thm:alpha_nu} that the active and phantom balance equations read 
\begin{align}\label{eq:balance-a-p}
    \lambda p_0=\left(\theta+\frac{1}{\mathbb{W}(\nu)}\right)Q_a(\nu),
    \qquad
    \lambda(1-p_0)+\theta Q_a(\nu)
    =\left(\phi\nu+\frac{1}{\mathbb{W}(\nu)}\right)Q_p(\nu).
\end{align}
Solving these two equations gives
\begin{align*}
    Q_a(\nu)=\frac{\lambda p_0 \mathbb{W}(\nu)}{1+\theta \mathbb{W}(\nu)},
    \qquad
    Q_p(\nu)=\frac{\lambda \mathbb{W}(\nu)(1-p_0+\theta \mathbb{W}(\nu))}{(1+\theta \mathbb{W}(\nu))(1+x \mathbb{W}(\nu))},
\end{align*}
which establishes~\eqref{eq:AP_nu}. 

From Theorem~\ref{thm:alpha_nu}, $\mathbb{W}(\nu)$ is finite for every $\nu \geq 0$ and strictly decreases in $\nu$. 
Then, it is clear that $Q_a(\nu) = \lambda p_0 / (1/\mathbb{W}(\nu) + \theta)$ strictly decreases in $\nu$. Note that
\begin{align*}
    Q_p(\nu) = \frac{\lambda(1-p_0)+\theta Q_a(\nu)}{\phi\nu+1/\mathbb{W}(\nu)}
\end{align*}
has a positive, strictly decreasing numerator and a positive, strictly increasing denominator. Hence, $Q_p(\nu)$ strictly decreases in $\nu$.  In particular, when $\nu \to \infty$, the first balance equation gives $\theta Q_a(\nu)\leq\lambda p_0$. The expression for $Q_p(\nu)$ therefore implies
\begin{align*}
    0\leq Q_p(\nu)\leq\frac{\lambda}{\phi\nu},
    \qquad \nu>0.
\end{align*}
Since $\phi>0$, it follows that $\lim_{\nu\to\infty}Q_p(\nu)=0$.

To derive the expression for the slot utilization, recall that enrollments occur at rate
\begin{align*}
    \frac{Q_a(\nu)}{\mathbb{W}(\nu)} = \frac{\lambda p_0}{1+\theta \mathbb{W}(\nu)}.
\end{align*}
Each enrollment occupies one slot for a mean duration $1/\delta$. Little's law therefore gives
\begin{align*}
    U(\nu)=\frac{\lambda p_0}{K\delta(1+\theta \mathbb{W}(\nu))},
\end{align*}
which establishes~\eqref{eq:U}. Since $\mathbb{W}(\nu)$ strictly decreases in $\nu$, it is clear that $U(\nu)$ strictly increases in $\nu$. 
\qed

\subsection{Proof of Proposition~\ref{prop:optimal_review}}

Define enrollment throughput as 
\begin{align*}
    e := \frac{\lambda p_0}{1 + \theta \mathbb{W}(\nu)}.
\end{align*}
From Theorem~\ref{thm:alpha_nu}, $\mathbb{W}(\nu)$ strictly decreases in $\nu$. Hence, $e$ increases continuously and strictly from $\delta(K-\lambda/\gamma)$ to $K\gamma\delta/(\gamma+\delta)$ as $\nu$ increases from $0$ to $\infty$; see the last part of the proof of Theorem~\ref{thm:alpha_nu}. Thus, minimizing total cost over $\nu\geq0$ is equivalent to minimizing it over
\begin{align*}
    e\in\left[ \delta\left(K-\frac{\lambda}{\gamma}\right), \frac{K\gamma\delta}{\gamma+\delta}\right).
\end{align*}
Below, derivatives with respect to $e$ are evaluated along this equilibrium relationship.

From~\eqref{eq:AP_nu}, 
\begin{align}\label{eq:Qa-e}
    Q_a(\nu) = e \mathbb{W}(\nu).
\end{align}
From~\eqref{eq:capacity_const}, 
\begin{align*}
    Q_p(\nu) = K\gamma \mathbb{W}(\nu) - \left(1+\frac{\gamma}{\delta}\right) Q_a(\nu) = K\gamma \mathbb{W}(\nu) - \left(1+\frac{\gamma}{\delta}\right) e \mathbb{W}(\nu),
\end{align*}
where the second equality follows from~\eqref{eq:Qa-e}. Thus, 
\begin{align}\label{eq:Qp-e}
    Q_p(\nu) = \left[K \gamma - \left(1+\frac{\gamma}{\delta}\right) e\right] \mathbb{W}(\nu).
\end{align}
Adding~\eqref{eq:Qa-e} and~\eqref{eq:Qp-e} gives
\begin{align}\label{eq:Qtotal-e}
    Q_a(\nu)+Q_p(\nu) =\left(K\gamma-\frac{\gamma}{\delta}e\right) \mathbb{W}(\nu).
\end{align}

Next, we examine how each of the three components of the objective function~\eqref{eq:reverify_cost} varies with $e$. 

\underline{1st component of~\eqref{eq:reverify_cost}:} From~\eqref{eq:Qp-e},
\begin{align*}
    \frac{sQ_p(\nu)}{\mathbb{W}(\nu)} = s\left[K \gamma - \left(1+\frac{\gamma}{\delta}\right) e\right],
\end{align*}
which is linearly decreasing in $e$. 

\underline{2nd component of~\eqref{eq:reverify_cost}:} From~\eqref{eq:U}, 
\begin{align*}
    rK\left(1 - U(\nu)\right) = rK\left(1-\frac{e}{K\delta}\right),
\end{align*}
which is also linearly decreasing in $e$. 

\underline{3rd component of~\eqref{eq:reverify_cost}:} Write $x := \phi \nu$. Then
\begin{align*}
    c_\nu\nu\left(Q_a(\nu)+Q_p(\nu)\right) =\frac{c_\nu}{\phi} x\left(Q_a(\nu)+Q_p(\nu)\right).
\end{align*}
The steady-state flow balance for the entire waitlist is
\begin{align*}
    \lambda =\frac{Q_a(\nu)+Q_p(\nu)}{\mathbb{W}(\nu)} + xQ_p(\nu),
\end{align*}
because families leave the waitlist through either an offer
or verification. Using~\eqref{eq:Qtotal-e}, this becomes
\begin{align}\label{eq:xQp-e}
    xQ_p(\nu) = \lambda - K\gamma + \frac{\gamma}{\delta}e.
\end{align}
Since $e<K\gamma\delta/(\gamma+\delta)$, the quantity
$K\gamma-(1+\gamma/\delta)e$ is strictly positive.
Consequently, \eqref{eq:Qp-e}--\eqref{eq:xQp-e} imply
\begin{align*}
    x\left(Q_a(\nu)+Q_p(\nu)\right)
    =xQ_p(\nu) \frac{Q_a(\nu)+Q_p(\nu)}{Q_p(\nu)}
    =\frac{\left(\lambda-K\gamma+\frac{\gamma}{\delta}e\right)\left(K\gamma-\frac{\gamma}{\delta}e\right)}{K\gamma-\left(1+\frac{\gamma}{\delta}\right)e}.
\end{align*}
To expand this expression, observe that
\begin{align*}
    K\gamma-\frac{\gamma}{\delta}e
    &=\frac{K\gamma\delta}{\gamma+\delta}+\frac{\gamma}{\gamma+\delta}
      \left[K\gamma-\left(1+\frac{\gamma}{\delta}\right)e\right],\\
    \lambda-K\gamma+\frac{\gamma}{\delta}e
    &=\lambda-\frac{K\gamma\delta}{\gamma+\delta}-\frac{\gamma}{\gamma+\delta}
      \left[K\gamma-\left(1+\frac{\gamma}{\delta}\right)e\right].
\end{align*}
Substituting these two identities into the numerator,
expanding the product, and dividing by the denominator yields
\begin{align}\label{eq:xQtotal-e}
    x\left(Q_a(\nu)+Q_p(\nu)\right)
    =\frac{\frac{K\gamma\delta}{\gamma+\delta}\left(\lambda-\frac{K\gamma\delta}{\gamma+\delta}\right)}{K\gamma-\left(1+\frac{\gamma}{\delta}\right)e}
    +\frac{\gamma}{\gamma+\delta}\left(\lambda-\frac{2K\gamma\delta}{\gamma+\delta}\right)
    -\left(\frac{\gamma}{\gamma+\delta}\right)^2\left[K\gamma-\left(1+\frac{\gamma}{\delta}\right)e\right].
\end{align}
Therefore,
\begin{align*}
    \frac{d}{de}\left(x\left(Q_a(\nu)+Q_p(\nu)\right)\right)
    =\left(1+\frac{\gamma}{\delta}\right)\left[\frac{\frac{K\gamma\delta}{\gamma+\delta}\left(\lambda-\frac{K\gamma\delta}{\gamma+\delta}\right)}{\left[K\gamma-\left(1+\frac{\gamma}{\delta}\right)e\right]^2}+\left(\frac{\gamma}{\gamma+\delta}\right)^2\right] > 0,
\end{align*}
and
\begin{align*}
    \frac{d^2}{de^2}\left(x\left(Q_a(\nu)+Q_p(\nu)\right)\right)
    =\frac{2\left(1+\frac{\gamma}{\delta}\right)^2\frac{K\gamma\delta}{\gamma+\delta}\left(\lambda-\frac{K\gamma\delta}{\gamma+\delta}\right)}{\left[K\gamma-\left(1+\frac{\gamma}{\delta}\right)e\right]^3}>0.
\end{align*}
Both inequalities use $\lambda\geq\lambda p_0>K\gamma\delta/(\gamma+\delta)$. Hence, verification cost is strictly increasing and strictly convex in $e$.

The first two cost components of~\eqref{eq:reverify_cost} are affine in $e$, whereas the third component is strictly convex. Therefore, total cost~\eqref{eq:reverify_cost} is strictly convex in $e$. Moreover, \eqref{eq:xQtotal-e} implies $\mathrm{TC}\to\infty$, as $e\to K\gamma\delta/(\gamma+\delta)$. Continuity and strict convexity consequently imply that total cost has a unique minimizer on the stated interval. The strictly increasing relationship between $e$ and $\nu$ then gives a unique finite optimal verification frequency $\nu^*$. 

\medskip

\noindent\emph{Verification threshold.}
Rearranging~\eqref{eq:xQp-e} gives
\begin{align*}
    e=\delta\left(K-\frac{\lambda}{\gamma}\right)+\frac{\delta}{\gamma}xQ_p(\nu).
\end{align*}
Differentiating with respect to $x$ at $x=0$ yields
\begin{align*}
    \left.\frac{de}{dx}\right|_{x=0} = \frac{\delta}{\gamma}Q_p(0)>0.
\end{align*}
Since $U(\nu)=e/(K\delta)$, it also follows that
\begin{align*}
    \left.\frac{dU}{dx}\right|_{x=0} = \frac{Q_p(0)}{K\gamma}.
\end{align*}

Differentiating the three cost components at $x=0$ gives
\begin{align*}
    \left.\frac{dTC}{dx}\right|_{x=0}
    =-\left[s\left(1+\frac{\gamma}{\delta}\right)+\frac{r}{\delta}\right]\frac{\delta}{\gamma}Q_p(0)+\frac{c_\nu}{\phi}\left(Q_a(0)+Q_p(0)\right).
\end{align*}
Using
\begin{align*}
    Q_p(0) = \left(1-\alpha^*(0)\right)\left(Q_a(0)+Q_p(0)\right),
\end{align*}
we obtain
\begin{align*}
    \left.\frac{dTC}{dx}\right|_{x=0}=Q_p(0)\left[\frac{c_\nu}{\phi\left(1-\alpha^*(0)\right)}-s-\frac{r+s\delta}{\gamma}\right].
\end{align*}

Because total cost is strictly convex in $e$ and
$\left.de/dx\right|_{x=0}>0$, its minimizer is above the
lower endpoint if and only if this derivative is negative.
Thus,
\begin{align*}
    \nu^*>0 \quad\Longleftrightarrow\quad
    \phi>\frac{c_\nu}{\left(s+\frac{r+s\delta}{\gamma}\right)\left(1-\alpha^*(0)\right)}.
\end{align*}
Recalling $V:=s+(r+s\delta)/\gamma$, define
\begin{align*}
    \phi_{\min} := \frac{c_\nu}{V\left(1-\alpha^*(0)\right)}.
\end{align*}
Thus, $\nu^*>0$ if and only if $\phi>\phi_{\min}$.
If the inequality fails, including equality, strict
convexity implies that the unique optimum is $\nu^*=0$.
This proves the threshold result. The zero-verification
acceptance probability used here, from Theorem~\ref{thm:alpha_nu}, is given by 
\begin{align*}
    \alpha^*(0)=\frac{\delta(K\gamma-\lambda)}{\lambda\gamma}.
\end{align*}

\medskip
\noindent\emph{Comparative statics.}
Viewing total cost as a function of $e$, its derivative is
\begin{align*}
    \frac{dTC}{de} = -s\left(1+\frac{\gamma}{\delta}\right) - \frac{r}{\delta} + \frac{c_\nu}{\phi}\frac{d}{de}\left(x\left(Q_a(\nu)+Q_p(\nu)\right)\right).
\end{align*}
The equilibrium relationship between $e$ and $\nu$ is independent of $s$, $r$, and $c_\nu$. Therefore,
\begin{align*}
    \frac{\partial}{\partial s}\left(\frac{dTC}{de}\right)&=-\left(1+\frac{\gamma}{\delta}\right)<0,\\
    \frac{\partial}{\partial r}\left(\frac{dTC}{de}\right)&=-\frac{1}{\delta}<0,\\
    \frac{\partial}{\partial c_\nu}\left(\frac{dTC}{de}\right)&=\frac{1}{\phi}\frac{d}{de}\left(x\left(Q_a(\nu)+Q_p(\nu)\right)\right)
    >0.
\end{align*}
Since $dTC/de$ is strictly increasing in $e$, the unique minimizer of total cost is nondecreasing in $s$ and $r$ and nonincreasing in $c_\nu$.  Finally, the strictly increasing, cost-independent relationship between $e$ and $\nu$ transfers these comparative statics to $\nu^*$. The comparisons are strict while the optimum is positive.
\qed

\section{Proofs from Section~\ref{sec:pooling}}

\subsection{Proof of Lemma~\ref{lemma:transition}}

\noindent\emph{(i)}
The infant classroom is the model of \S\ref{subsec:impatience} with parameters $(K_I,\lambda_I,\delta_I)$. Theorem~\ref{thm:fixed_point} therefore gives a unique equilibrium acceptance probability $\alpha_I^*$. Then, infant enrollments occur at rate $\lambda_I\alpha_I^*$, including when $\alpha_I^*=0$. Since infant slot utilization is bounded, departures have the same long-run rate. Independent move-up decisions thus give $g=\sigma\lambda_I\alpha_I^*$, which establishes~\eqref{eq:grad_flow}.

\noindent\emph{(ii)}
Move-ups occupy a mean of $g/\delta_T$ toddler slots, so feasibility requires $0\leq g\leq K_T\delta_T$. Fix such a $g$ and an external acceptance probability $\alpha$. To derive the effective clearance rate (the rate at which external offers are resolved) $\mu_T(\alpha)$, suppose that the external waitlist is always nonempty. Every toddler slot is then either occupied or holds an outstanding external offer. Little's law gives $\mu_T(\alpha)/\gamma$ outstanding offers and $(g+\alpha\mu_T(\alpha))/\delta_T$ occupied slots on average. Hence,
\begin{align*}
    \frac{\mu_T(\alpha)}{\gamma} + \frac{g+\alpha\mu_T(\alpha)}{\delta_T} = K_T,
\end{align*}
and therefore,
\begin{align}\label{eq:mu-T}
    \mu_T(\alpha)
    =\frac{\gamma(K_T\delta_T-g)}{\alpha\gamma+\delta_T}
    =\frac{(K_T-g/\delta_T)\gamma\delta_T}
           {\alpha\gamma+\delta_T}.
\end{align}
Thus, the effective clearance rate has the form in~\eqref{eq:clearing_rate}, with effective capacity $K_T-g/\delta_T$.

Given that the external toddler waitlist has arrival rate $\lambda_T$ and clearance rate $\mu_T(\alpha)$ whenever nonempty, the mean waiting time is given by
\begin{align*}
    \mathbb{W}_T(\alpha)
    =\frac{1}{(\mu_T(\alpha)-\lambda_T)^+}
    =\frac{\alpha\gamma + \delta_T}{\left(\gamma\,(K_T\delta_T - g) - \alpha\gamma\lambda_T - \delta_T \lambda_T\right)^{+}},
\end{align*}
establishing~\eqref{eq:wait_T}. This is~\eqref{eq:alpha_wait} with $(K,\lambda,\delta)$ replaced by $(K_T-g/\delta_T,\lambda_T,\delta_T)$.

For $g<K_T\delta_T$, the proofs of Theorem~\ref{thm:fixed_point} and Lemma~\ref{lemma:alpha-star} apply with $(K,\lambda,\delta)$ replaced by $(K_T-g/\delta_T,\lambda_T,\delta_T)$, since capacity enters those arguments only through the clearance rate. Hence, the toddler fixed point is unique and satisfies
\begin{align*}
    \alpha_T^* = \frac{p_0}{1+\theta\mathbb{W}_T(\alpha_T^*)}.
\end{align*}
Define
\begin{align*}
    g_c:=\delta_T\left(K_T-\frac{\lambda_T}{\gamma}\right).
\end{align*}
If $g<g_c$, equivalently $\lambda_T<\mu_T(0)$, then
\begin{align*}
    0<\alpha_T^*<p_0,
    \qquad
    \lambda_T<\mu_T(\alpha_T^*),
    \qquad
    \mathbb{W}_T(\alpha_T^*)<\infty.
\end{align*}
If $g\geq g_c$, then $\alpha_T^*=0$ and $\mathbb{W}_T(\alpha_T^*)=\infty$. These conclusions also hold at $g=K_T\delta_T$, where $\mu_T(\alpha)\equiv0$.

As in part~(i), external toddler enrollments occur at rate $\lambda_T\alpha_T^*$. Total toddler enrollments therefore occur at rate $g+\lambda_T\alpha_T^*$. When this rate is positive, the fraction attributable to move-ups is
\begin{align*}
    \frac{g}{g+\lambda_T\alpha_T^*}.
\end{align*}
Little's law gives a mean of $(g+\lambda_T\alpha_T^*)/\delta_T$ occupied toddler slots and utilization
\begin{align*}
    U_T=\frac{g+\lambda_T\alpha_T^*}{K_T\delta_T},
\end{align*}
establishing the occupancy expressions in~\eqref{eq:U_T}.

It remains to establish the comparative statics in $g$. We consider two cases. 

\textbf{Case 1:} $g<g_c$. Write $\mu:=\mu_T(\alpha_T^*)>\lambda_T$. The fixed-point equation and waiting-time formula (Theorem~\ref{thm:fixed_point}) give
\begin{align*}
    \alpha_T^*=p_0T_T(\mu),
    \quad \mbox{ where } \ 
    T_T(\mu):=\frac{\mu-\lambda_T}{\mu-\lambda_T+\theta}.
\end{align*}
The function $T_T$ is continuous and strictly increasing, with $T_T(\mu)\to0$ as $\mu\downarrow\lambda_T$. Substituting into~\eqref{eq:mu-T} with $\mu_T(\alpha_T^*) = \mu$ and rearranging gives
\begin{align*}
    g=K_T\delta_T-\frac{\delta_T}{\gamma}\mu-p_0\mu T_T(\mu).
\end{align*}
The right-hand side is continuous and strictly decreasing in $\mu$, with limits $g_c$ as $\mu\downarrow\lambda_T$ and $-\infty$ as $\mu\to\infty$. Consequently, $\mu$ and $\alpha_T^*=p_0T_T(\mu)$ are continuous and strictly decreasing in $g$, while $\mathbb{W}_T(\alpha_T^*)=1/(\mu-\lambda_T)$ strictly increases.

Moreover,
\begin{align*}
    g+\lambda_T\alpha_T^* = K_T\delta_T-\frac{\delta_T}{\gamma}\mu - p_0(\mu-\lambda_T)T_T(\mu).
\end{align*}
Both subtracted terms strictly increase in $\mu>\lambda_T$. Thus, the total toddler enrollment rate, and hence slot utilization $U_T$, strictly increases in $g$.

\textbf{Case 2:} $g\geq g_c$. Here, $\alpha_T^*=0$ and $U_T=g/(K_T\delta_T)$, so utilization again strictly increases in $g$. If $g_c>0$, then as $g\uparrow g_c$, \eqref{eq:mu-T} forces $\mu\downarrow\lambda_T$ and $\alpha_T^*\downarrow0$, joining the two regimes continuously. If $g_c\leq0$, every feasible $g$ lies in the zero-acceptance regime. Therefore, $\alpha_T^*$ is continuous in $g$, strictly decreasing in the stable regime and zero thereafter, whereas total toddler enrollments and slot utilization $U_T$ are continuous and strictly increasing throughout the feasible range.
\qed

\subsection{Proof of Lemma~\ref{lemma:pooling_alpha}}

Throughout the proof, $f\geq0$ is treated as a continuous variable, and $\lambda_I<K_I\gamma$, $g(0)<K_T\delta_T$, and the external toddler wait is finite at $f=\hat f$, as assumed. Let
\begin{align*}
    \mu_I(f) := \frac{(K_I+f)\gamma\delta_I}{\alpha_I^*(f)\gamma+\delta_I},
    \qquad
    T_I(\mu):=\frac{\mu-\lambda_I}{\mu-\lambda_I+\theta},
    \quad \mu>\lambda_I.
\end{align*}
Under $\lambda_I<K_I\gamma$, by Lemma~\ref{lemma:alpha-star}, $\mu_I(f)>\lambda_I$ for every $f\geq0$. 

\noindent\emph{(i)}
The infant classroom has $K_I+f$ slots, so its fixed-point equation is
$\alpha_I^*(f)=p_0T_I(\mu_I(f))$. Rearranging the definition of $\mu_I(f)$ gives
\begin{align*}
    K_I+f = \frac{\mu_I(f)}{\gamma} + \frac{p_0\mu_I(f)T_I(\mu_I(f))}{\delta_I}.
\end{align*}
The right-hand side is continuous and strictly increasing in $\mu_I(f)$ on $(\lambda_I,\infty)$, with limits $\lambda_I/\gamma$ and $\infty$ at the two endpoints. Thus, $\mu_I(f)$ is continuous and strictly increasing in $f$. Consequently, $\alpha_I^*(f)$ is continuous and strictly increasing, while
\begin{align*}
    \mathbb{W}_I(f)=\frac{1}{\mu_I(f)-\lambda_I}
\end{align*}
is continuous and strictly decreasing. Moreover,~\eqref{eq:alpha_u} gives
\begin{align*}
    (K_I+f)U_I(f)
    =\frac{\lambda_I\alpha_I^*(f)}{\delta_I}.
\end{align*}
Hence, the expected number of occupied infant slots $(K_I+f) U_I(f)$ strictly increases in $f$.

\noindent\emph{(ii)} By~\eqref{eq:grad_flow}, move-ups leave the infant classroom at rate $g(f)=\sigma\lambda_I\alpha_I^*(f)$, so Lemma~\ref{lemma:transition}(ii) applies with $g$ replaced by $g(f)$ wherever $g(f)\leq K_T\delta_T$. By part~(i), $g(f)$ is continuous and nondecreasing in $f$, and strictly increasing if $\sigma>0$. By the proof of Lemma~\ref{lemma:transition}(ii), $\alpha_T^*$ is continuous in $g$, strictly decreasing while
\begin{align*}
    \frac{g}{\delta_T}+\frac{\lambda_T}{\gamma}<K_T,
\end{align*}
that is, while the external toddler wait is finite, and zero thereafter, and the toddler enrollment rate $g+\lambda_T\alpha_T^*$ strictly increases in $g$. Composing the two, $\alpha_T^*(f)$ is continuous and nonincreasing in $f$, and the toddler enrollment rate $g(f)+\lambda_T\alpha_T^*(f)$ is continuous and nondecreasing in $f$. If $\sigma>0$, then $U_T(f)=(g(f)+\lambda_T\alpha_T^*(f))/(K_T\delta_T)$ strictly increases in $f$, and $\alpha_T^*(f)$ strictly decreases in $f$ wherever the external toddler wait is finite, which, as shown at the end of the proof, includes all of $[0,\hat f]$.

\medskip
Finally, we show that~\eqref{eq:lendable} has a unique solution, so that the interval $[0,\hat f]$ in the statement is well defined, that $g(0)<K_T\delta_T$ is equivalent to $\hat f>0$, and that the external toddler wait is finite on all of $[0,\hat f]$. Let
\begin{equation*}
    \ell(f):=f+\frac{g(f)+\lambda_T\alpha_T^*(f)}{\delta_T},\qquad f\geq0,
\end{equation*}
the $f$ lent slots plus the mean number of occupied toddler slots, so that~\eqref{eq:lendable} reads $\ell(\hat f)=K_T$. The steady-state requirement $g(f)\leq K_T\delta_T$ holds on an interval containing zero, because $g(0)<K_T\delta_T$ and $g$ is continuous and nondecreasing; at the right endpoint of that interval, if there is one, $g(f)=K_T\delta_T$, so move-ups occupy all $K_T$ toddler slots and $\ell(f)\geq K_T$. On that interval, by part~(ii), the second term of $\ell$ is continuous and nondecreasing in $f$, so $\ell$ is continuous and strictly increasing, with $\ell(f)\geq f$. Moreover, $\ell(0)<K_T$: if the external toddler wait is finite at $f=0$, then $\lambda_T<\mu_T(\alpha_T^*(0))$ rearranges to $\ell(0)+\lambda_T/\gamma<K_T$, and otherwise $\alpha_T^*(0)=0$ and $\ell(0)=g(0)/\delta_T<K_T$. The intermediate value theorem therefore yields a solution $\hat f\in(0,K_T]$ of $\ell(\hat f)=K_T$, strict monotonicity makes it unique, and $\ell(f)\leq K_T$ precisely when $f\leq\hat f$. Conversely, if $g(0)\geq K_T\delta_T$, then $g(f)\geq K_T\delta_T$ for every $f\geq0$, so move-ups alone would occupy all $K_T$ toddler slots and no $f>0$ satisfies $\ell(f)\leq K_T$. Hence, $g(0)<K_T\delta_T$ is equivalent to $\hat f>0$.

The external toddler wait is finite at $f=\hat f$ whenever $\hat f>\lambda_T/\gamma$: were it infinite, the condition of Case~2 in the proof of Lemma~\ref{lemma:transition} would hold at $g=g(\hat f)$, because the wait is finite in Case~1, so that $\alpha_T^*(\hat f)=0$ and $g(\hat f)/\delta_T+\lambda_T/\gamma\geq K_T$, and $\ell(\hat f)=K_T$ would read $\hat f=K_T-g(\hat f)/\delta_T\leq\lambda_T/\gamma$. Because that wait is finite, the condition of Case~1 in the same proof holds at $g=g(\hat f)$ and, because $g$ is nondecreasing, at every $f\in[0,\hat f]$, so the external toddler wait is finite on all of $[0,\hat f]$.
\qed

\subsection{Proof of Proposition~\ref{prop:optimal_pooling}}

Let  $\check f := K_I - \lambda_I\alpha_I^*(0)/\delta_I$ be the maximum number of infant slots that can be lent to the toddler classroom. 

Throughout the proof, $f$ is treated as a continuous variable on $[-\check f,\hat f]$, and $\lambda_I<K_I\gamma$, $g(0)<K_T\delta_T$, and the external toddler wait is finite at $f=\hat f$, as assumed, so that, by the proof of Lemma~\ref{lemma:pooling_alpha}, $\hat f>0$ and the external toddler wait is finite on all of $[0,\hat f]$. 
Assume also that $C$ is continuous and nondecreasing on the feasible interval and strictly convex on $[0,\hat f]$. 
By~\eqref{eq:alpha_u}, $\check{f}=K_I(1-U_I(0))>0$. 
For $f<0$, set $\alpha_I^*(f):=\alpha_I^*(0)$ and $g(f):=g(0)$, and evaluate $\alpha_T^*(f)$ using toddler capacity $K_T-f$. These definitions agree with the positive branch at zero. The toddler wait is finite throughout the feasible interval: on $[0,\hat f]$, move-up flow is no larger than at $\hat f$; on $[-\check f,0]$, it is fixed while toddler capacity increases.

We first establish a useful property of the acceptance probabilities. For either classroom $j\in\{I,T\}$, combining the fixed-point and clearance identities expresses effective capacity as
\begin{align*}
    \left(\lambda_j+\frac{\theta\alpha}{p_0-\alpha}\right)\left(\frac{1}{\gamma}+\frac{\alpha}{\delta_j}\right),
    \qquad 0<\alpha<p_0.
\end{align*}
This expression has a positive first derivative and the strictly positive second derivative with respect to $\alpha$ 
\begin{align*}
    \frac{2\theta p_0(1/\gamma+p_0/\delta_j)}{(p_0-\alpha)^3}>0.
\end{align*}
Hence, its inverse, the equilibrium acceptance probability, is continuous, strictly increasing, and strictly concave in effective capacity.

Write $E(f):=g(f)+\lambda_T\alpha_T^*(f)$ for the total toddler enrollment rate. Substituting~\eqref{eq:grad_flow} into~\eqref{eq:pooling_objective} gives
\begin{align}\label{eq:TC-alt}
    \mathrm{TC}(f) = s(\lambda_I+\lambda_T) + r(K_I+K_T) - (V_I-\sigma V_T)\lambda_I\alpha_I^*(f) - V_T E(f) + C(f),
\end{align}
where
\begin{align*}
    V_I-\sigma V_T=s(1-\sigma)+\frac{r}{\delta_I}\geq0.
\end{align*}

For $f\leq0$, the infant acceptance probability and move-up flow are constant, while effective toddler capacity is $K_T-f-g(0)/\delta_T$. Thus, $\alpha_T^*(f)$ and $E(f)$ decrease continuously and strictly as $f$ increases. Since $V_T>0$ and $C$ is nondecreasing, $\mathrm{TC}$ is continuous and strictly increasing on $[-\check f,0]$. Consequently, $-\check f$ is its unique minimizer on this interval, and $\mathrm{TC}(-\check f)<\mathrm{TC}(0)$.

For $f\geq0$, infant capacity is $K_I+f$, so $\alpha_I^*(f)$ is strictly increasing and strictly concave. Hence, $g(f)=\sigma\lambda_I\alpha_I^*(f)$ is concave and nondecreasing. Effective toddler capacity is $K_T-g/\delta_T$, so toddler acceptance is concave in $g$. The total toddler enrollment rate $g+\lambda_T\alpha_T^*$ is therefore concave in $g$ and, by Lemma~\ref{lemma:transition}, increasing in $g$. Composing it with $g(f)$ shows that $E(f)$ is concave in $f$.

The expression for $\mathrm{TC}$ in~\eqref{eq:TC-alt} is consequently a sum of convex terms on $[0,\hat f]$, including the strictly convex function $C$. Thus, $\mathrm{TC}$ is strictly convex there and has a unique minimizer, denoted by $f_+$.

The objective is continuous on the compact feasible interval, so a global optimum exists. Every global optimum equals either $-\check f$ or $f_+$, whichever has lower cost; both are optimal if their costs coincide. Since $\mathrm{TC}(-\check f)<\mathrm{TC}(0)$, any globally optimal $f_+$ must be strictly positive.

\noindent\emph{(i)}
Differentiating from the right at zero gives
\begin{align*}
    \mathrm{TC}'(0+) = C'(0) - \lambda_I\alpha_I^{*\prime}(0)V_I - \lambda_T\alpha_T^{*\prime}(0)V_T,
\end{align*}
where all derivatives at zero are right derivatives. Strict convexity on $[0,\hat f]$ implies
\begin{align*}
    f_+>0
    \quad\Longleftrightarrow\quad
    \lambda_I\alpha_I^{*\prime}(0)V_I
    +\lambda_T\alpha_T^{*\prime}(0)V_T>C'(0).
\end{align*}
Thus, \eqref{eq:pooling_threshold} is necessary for a positive global optimum. If it fails, $f_+=0$ and the unique global optimum is $-\check f$. If it holds, global optimality is determined by comparing $\mathrm{TC}(f_+)$ with $\mathrm{TC}(-\check f)$.

\noindent\emph{(ii)}
The feasible interval is independent of $s$, $r$, and $C$. On $[0,\hat f]$, the coefficients of $s$ and $r$ in the objective are, respectively,
\begin{align*}
\begin{aligned}
    &\lambda_I+\lambda_T -(1-\sigma)\lambda_I\alpha_I^*(f)-E(f),\\
    &K_I+K_T-\frac{\lambda_I\alpha_I^*(f)}{\delta_I}-\frac{E(f)}{\delta_T}.
\end{aligned}
\end{align*}
Both strictly decrease, meaning that infant enrollments strictly increase, while toddler enrollments are nondecreasing and strictly increase when $\sigma>0$.

Increasing either $s$ or $r$ therefore adds a strictly decreasing function of $f$ to the objective. This cannot move its unique minimizer to the left: any smaller choice had strictly higher cost before the change and receives a larger cost increase. Hence, $f_+$ is nondecreasing in $s$ and $r$.

Finally, if two admissible flexibility costs satisfy $C_2'\geq C_1'$ on $(0,\hat{f})$, then $C_2-C_1$ is nondecreasing on $[0,\hat{f}]$. Adding this difference cannot move the unique minimizer to the right: any larger choice had strictly higher cost under $C_1$ and receives a weakly larger cost change. Thus, $f_+$ is nonincreasing in marginal flexibility cost.
\qed

\end{document}